\documentclass[11pt]{amsart}
\usepackage{amsmath,amssymb,mathrsfs,hyperref,color}

\usepackage{subfigure}
\usepackage{float} 
\usepackage{times}
\usepackage{tikz}
\usepackage{cases}
\usetikzlibrary{intersections}
\usepackage{paralist}
\usepackage{xcolor}

\usetikzlibrary{shadings}
\usetikzlibrary[intersections]
\usetikzlibrary{arrows,shapes}

\makeatletter
\def\setliststart#1{\setcounter{\@listctr}{#1}%
  \addtocounter{\@listctr}{-1}}
\makeatother

\makeatletter
\@addtoreset{figure}{section}
\makeatother

\usepackage{calc}
 \newtheorem{The}{Theorem}[section]
 \newtheorem{Cor}[The]{Corollary}
 \newtheorem{Lem}[The]{Lemma}
 \newtheorem{Pro}[The]{Proposition}
 \theoremstyle{definition}
 \newtheorem{defn}[The]{Definition}
 \newtheorem{Result}[The]{\textbf{Main Result}}
 \theoremstyle{remark}
 \newtheorem{Rem}[The]{Remark}
 
 \numberwithin{equation}{section}

\newcommand{\R}{\mathbb{R}}

\newcommand{\N}{\mathbb{N}}

\newcommand{\SING}{\mbox{\rm Sing$\,(u)$}}
\newcommand{\CUT}{\mbox{\rm Cut}\,(u)}
\newcommand{\BSING}{\overline{\SING}}
\newcommand{\IU}{\mathcal{I}(u)}

\title[Global generalized characteristics for the Dirichlet problem]{{\small Global generalized characteristics\\for the Dirichlet problem for Hamilton-Jacobi equations\\at a supercritical energy level}}
\author{Piermarco Cannarsa \and Wei Cheng \and Marco Mazzola \and Kaizhi Wang}
\address{Dipartimento di Matematica, Universit\`a di Roma ``Tor Vergata'',
Via della Ricerca Scientifica 1, 00133 Roma, Italy}
\email{cannarsa@mat.uniroma2.it}
\address{Department of Mathematics, Nanjing University, Nanjing 210093, China}
\email{chengwei@nju.edu.cn}
\address{IMJ-PRG, Sorbonne Universit\'e, CNRS, case 247, 4 place Jussieu, Paris, France}
\email{marco.mazzola@imj-prg.fr}
\address{School of Mathematical Sciences, Shanghai Jiao Tong University, Shanghai 200240, China}
\email{kzwang@sjtu.edu.cn}
\date{\today}
\begin{document}
\maketitle

\begin{abstract}
	We study the nonhomogeneous Dirichlet problem for first order Hamilton-Jacobi equations associated with Tonelli Hamiltonians on a bounded domain $\Omega$ of $\R^n$ assuming the energy level to be supercritical. First, we show that the viscosity (weak KAM) solution of such a problem is Lipschitz continuous and locally semiconcave in $\Omega$. Then, we analyse the singular set of a solution showing that singularities propagate along suitable curves, the so-called generalized characteristics, and that such curves stay singular unless they reach the boundary of $\Omega$. Moreover, we prove that the latter is never the case for mechanical systems and that singular generalized characteristics converge to a critical point of the solution in finite or infinite time. Finally, under stronger assumptionsfor the domain and Dirichlet data, we are able to conclude that solutions are globally semiconcave and semiconvex near the boundary.
\end{abstract}

\tableofcontents

\section{Introduction}
Let $\Omega\subset \R^n$ be a bounded Lipschitz domain and let $H:\R^n\times\R^n\to\R$ be a Tonelli Hamiltonian satisfying Fathi-Maderna's conditions (see \cite{Fathi-Maderna} and section~\ref{se:DP} below). We consider the Dirichlet boundary-value problem for a first-order Hamilton-Jacobi equation
\begin{equation}\label{eq:HJ_BP1}
	\begin{cases}
	H(x,Du)=0 \quad \text{in}\ \Omega, \\
	u\big|_{\partial\Omega}=g,
	\end{cases}
\end{equation}
where $g$ is a given continuous function on $\partial\Omega$. 
The purpose of this paper is to study the propagation of singularities of the viscosity solution $u$ of \eqref{eq:HJ_BP1} and the structure of the cut locus of $u$, $\CUT$, as well as the singular set of $u$, $\SING$. Our interest in these problems has several  motivations, some of which are described below.

First, in  weak KAM theory (\cite{Fathi-book,Fathi-Maderna,Fathi-Siconolfi2004,Contreras}), one considers the Hamilton-Jacobi equation
\begin{equation}\label{eq:weak_KAM}
	H(x,Du)=c,\quad x\in M,
\end{equation}
where $M$ is a  smooth connected manifold without boundary and $c\in\R$ is  Ma\~n\'e's critical value (\cite{Mane}). A function $u$ is said to be a weak KAM solution of \eqref{eq:weak_KAM}  if it is a fixed point of the map $u\mapsto T^-_tu+ct$ for all $t>0$, where $T^-_t$ denotes the negative type Lax-Oleinik operator (see, for instance, \cite{Fathi-book}). On the other hand, if one looks at the relevant projected Aubry set $\mathcal{A}$ as a virtual boundary (see, for instance, \cite{Ishii-Mitake} and \cite{Cui}), then \eqref{eq:weak_KAM} can be also understood as a Dirichlet problem by taking $\Omega=M\setminus\mathcal{A}$ in \eqref{eq:HJ_BP1}. Therefore, it is useful to deal with \eqref{eq:HJ_BP1} under more general boundary conditions, and, it will help us to obtain more information on the relation between the regularity properties of $\partial \Omega$ and  $g$  and the structure of $\CUT$. 

Second, the global propagation of singularities for the eikonal equation 
$$
|Du|^2-1=0
$$
with homogeneous boundary conditions
$$
u|_{\partial \Omega}=0
$$
was proved in \cite{ACNS} by ``quantitative" methods. For weak KAM solutions of equation \eqref{eq:weak_KAM} on the whole space, the analogous result of global propagation was obtained in
 \cite{Cannarsa-Cheng3} by a ``qualitative" approach. Indeed, in \cite{Cannarsa-Cheng3}, the problem was solved by using the positive type Lax-Oleinik semigroup which gives an intrinsic explanation of the propagation of singularities only according to the associated system of characteristics. Later, in \cite{Cannarsa-Cheng-Fathi}, the method was applied to obtain  topological results for $\CUT$ and $\SING$ such as the homotopy equivalence between the complement of the projected Aubry set of $u$ and $\CUT$ or $\SING$, and the local path-connectedness of $\CUT$ and $\SING$. Indeed, the method developed in \cite{Cannarsa-Cheng3} and \cite{Cannarsa-Cheng-Fathi} can be applied to various kinds of problems as this paper will confirm.

Third,  problem \eqref{eq:HJ_BP1} is closely related to  optimal exit time problems in  control theory. 
It can also be regarded as a first step towards the analysis of constrained optimal control problems from the point of view of weak KAM theory.

In order to apply the methods developed in \cite{Cannarsa-Cheng3} and \cite{Cannarsa-Cheng-Fathi} to \eqref{eq:HJ_BP1}, we need a representation formula for the solution of \eqref{eq:HJ_BP1}.  For any $x,y\in\overline{\Omega}$ and any $s<t$, we define the set of admissible arcs from $x$ to $y$ as
\begin{align*}
	\Gamma^{s,t}_{x,y}(\overline{\Omega})=\{\xi\in W^{1,1}([s,t];\R^n)~:~\xi(\tau)\in\overline{\Omega}\,, \forall\ \tau\in[s,t];\ \xi(s)=x;\ \xi(t)=y\}.
\end{align*}
For any $x,y\in\overline{\Omega}$ and $t>0$, we define the fundamental solution $A^{\Omega}_t(x,y)$ relative to $\overline{\Omega}$, Ma\~n\'e's  potential $\Phi^{\Omega}_L(x,y)$ relative to $\overline{\Omega}$, and critical value $c_{\Omega}(L)$ relative to $\overline{\Omega}$, by 
\begin{align*}
A^{\Omega}_t(x,y)&:=\inf_{\xi\in\Gamma^{0,t}_{x,y}(\overline{\Omega})}\int^t_0L(\xi(s),\dot{\xi}(s))\ ds,\\
\Phi^{\Omega}_L(x,y)&:=\inf_{t>0}A_t^{\Omega}(x,y),\quad c_{\Omega}(L):=-\inf_{t>0,x\in\overline{\Omega}}\frac 1tA^{\Omega}_t(x,x).
\end{align*}

Let $u$ be the value function of the following problem:
\begin{equation}\label{eq:represent_formulae2}
	u(x)=\inf_{y\in\partial\Omega}\{g(y)+\Phi^{\Omega}_L(y,x)\},\quad x\in\overline{\Omega},
\end{equation}
where $g:\partial\Omega\to\R$ is a  continuous function satisfying
\begin{equation}\label{eq:supcritical1i}
	g(x)-g(y)\leqslant\Phi^{\Omega}_L(y,x),\quad \forall x,y\in\partial\Omega.
\end{equation}
 
Under the assumption that
\begin{equation}\label{eq:supcritical2i}
	c_{\Omega}(L)<0,
\end{equation}
one can show that $u$ in \eqref{eq:represent_formulae2} is a locally semiconcave viscosity solution of \eqref{eq:HJ_BP1} and it is Lipschitz  continuous on $\overline{\Omega}$.  We obtain the following dichotomy:

\begin{Result}
	Let $x_0\in\CUT$. Then, we have
	\begin{enumerate}[\rm (a)]
	\item either, there exists a generalized characteristic $\mathbf{x}:[0,+\infty)\to\Omega$ starting from $\mathbf{x}(0)=x_0$ such that $\mathbf{x}(s)\in\SING$ for all $s\in[0,+\infty)$,
	\item or, there exist $T>0$ and a generalized characteristic $\mathbf{x}:[0,T)\to\Omega$ starting from $\mathbf{x}(0)=x_0$ such that $\mathbf{x}(s)\in\SING$ for all $s\in[0,T)$, and a sequence of positive real numbers $\{s_k\}$ such that
	$$
	\lim_{k\to\infty}s_k=T,\quad\text{and}\quad\lim_{k\to\infty}d_{\partial\Omega}(\mathbf{x}(s_k))=0.
	$$
	\end{enumerate}
\end{Result}

Now, as is well known, local semiconcavity is not enough to obtain global propagation of singularities for $u$. So, we have to specialize our analysis as follows. 

For mechanical systems,  the associated generalized characteristics system has a unique forward solution, i.e., there exists a unique generalized characteristic from any starting point. In addition, the semi-flow generated by generalized characteristics systems has certain monotonicity properties. Therefore, we can obtain a generalized characteristic  on $[0,+\infty)$ which consists of singular points for $u$ if the starting point is a cut point. More precisely, let the Lagrangian be of the form
\begin{align*}
	L(x,v)=\frac 12\langle A(x)v,v\rangle-\langle DS(x),v\rangle-V(x),\quad (x,v)\in\R^n\times\R^n,
\end{align*}
where $A(x)$ is a symmetric and positive definite matrix $C^2$ depending on $x$ and $S$ (resp. $V$) is a $C^3$ (resp. $C^2$) function on $\R^n$. Let us further assume that
\begin{align*}
	\ \max_{x\in\overline{\Omega}}V(x)<0\quad \text{and}\quad g+S\ \text{is constant on}\ \partial\Omega.
\end{align*}
Let $L_0(x,v)=L(x,v)+\langle DS(x),v\rangle=\frac 12\langle A(x)v,v\rangle-V(x)$. Denote by $H_0$ the Hamiltonian associated with $L_0$. We obtain that

\begin{Result}
Let $v=u+S$. If $x_0\in\CUT$, then there exists a unique generalized characteristic $\mathbf{x}:[0,+\infty)\to\Omega$ with $\mathbf{x}(0)=x_0$ for $H_0(x,Dv)=0$, i.e.,  is a Lipschitz curve $\mathbf{x}$ with $\mathbf{x}(0)=x_0$ such that
\begin{align*}
	\dot{\mathbf{x}}^+(s)\in A^{-1}(\mathbf{x}(s))D^+v(\mathbf{x}(s)),\quad \forall s\in[0,+\infty).
\end{align*}
Moreover,  $\mathbf{x}(s)\in\SING$ for all $s\in[0,+\infty)$.
\end{Result}

As a consequence, we can recover all the topological results in \cite{Cannarsa-Cheng-Fathi} in this case.

\begin{Result}
	The inclusion $\SING\subset\CUT\subset\BSING\cap\Omega\subset\Omega$ are all homotopy equivalences. Moreover, for every connected component $C$ of $\Omega$ the three intersections $\SING\cap C$, $\CUT\cap C$, and $\BSING\cap C$ are path-connected.
\end{Result}

\begin{Result}
	The spaces $\SING$ and $\CUT$ are locally contractible, i.e., for every $x\in\SING$ (resp. $x\in\CUT$) and every neighborhood $V$ of $x$ in $\SING$ (resp. $\CUT$), we can find a neighborhood $W$ of $x$ in $\SING$ (resp. $\CUT$), such that $W\subset V$ and $W$ in null-homotopic in $V$. 
	
	Therefore, $\SING$ and $\CUT$ are locally path connected.
\end{Result}

On the other hand, for general Tonelli systems, we need to restrict the analysis to smoother data, that is, $\partial\Omega$ of class $C^2$ and $g$ of class $C^{1,1}$ on $\partial \Omega$. Under such conditions, $u$ can be proved to be smooth in a neighborhood of the boundary. Therefore, by using the method in \cite{Cannarsa-Cheng3} again, we get our results on the global propagation of singularities.

\begin{Result}
	Let $\Omega\subset\R^n$ be a bounded domain with $C^2$ boundary, let $L$ be a Tonelli Lagrangian satisfying  $L\geqslant\alpha >0$ and let $g$ satisfy ({\bf G1}),({\bf G2})\footnote{For precise statements of conditions ({\bf G1}) and ({\bf G2}), please see the beginning of Section 5.}. If $x_0\in\CUT$, then there exists a generalized characteristic $\mathbf{x}:[0,+\infty)\to\Omega$ starting from $\mathbf{x}(0)=x_0$ such that $\mathbf{x}(s)\in\SING$ for all $s\in[0,+\infty)$.
\end{Result}

The rest of the paper is organized as follows. Section 2 gives the basic definitions and preliminaries required for our subsequent work. In Section 3, we discuss the properties of the value function $u$ of  problem \eqref{eq:represent_formulae2}, and its relation to exit time problems. Section 4 is the main part of the present paper. It consists of two parts: for general Tonelli Lagrangian systems, we provide a result on propagation of singularities for $u$, but we cannot exclude the possibility that the singularities approach the boundary; for mechanical systems we obtain a result on global propagation of singularities for $u$ and more information on the topology of $\CUT$. In Section 5, for general Tonelli Lagrangian systems, under certain additional conditions we can get global semiconcavity of $u$ on $\overline{\Omega}$ and local semiconvexity of $u$ near the boundary, which imply the global propagation of singularities for $u$. 

\vskip.2cm
\noindent {\bf Notations}. We denote by $|\cdot|$ the Euclidean norm in $\R^n$, by $\langle\cdot,\cdot\rangle$ the inner product, by $\partial\Omega$ the boundary of $\Omega$ taken with respect to the standard topology of $\R^n$,
by $B(x,r)$ the open ball of center $x$ and radius $r>0$, by $d_{\partial \Omega}(x)$ the Euclidean distance between a point $x$ and $\partial \Omega$, by $d(S_1,S_2)$ the Euclidean distance between two subsets $S_1$ and $S_2$ of $\R^n$, by $\mathrm{co}\, S$ the convex hull of a subset $S$ of $\R^n$, by $[x,y]$ the segment with endpoints $x$, $y$, for any $x$, $y\in\R^n$, by $\mathrm{Lip}(u)$ a Lipschitz constant of a Lipschitz function $u$,
by $f_x(x,y)$, $\frac{\partial f}{\partial x}(x,y)$ or $D_xf(x,y)$ the partial derivative of a function $f(x,y)$ with respect to the variable $x$.

\section{Dirichlet problem}\label{se:DP}

Throughout this paper, we assume that $L:\R^n\times\R^n\to\R$, $(x,v)\mapsto L(x,v)$, is a $C^2$ function satisfying the following conditions.
\begin{enumerate}[(\bf{L}1)]
\item {\it Convexity}: the Hessian $\frac{\partial^2L}{\partial v^2}(x,v)$ is positive definite for all $(x,v)\in \R^n\times \R^n$.
\item {\it Superlinearity}: there exist two nondecreasing superlinear  functions $\theta_1,\theta_2:[0,+\infty)\to[0,+\infty)$ and a constant $c_0>0$ such that 
$$
\theta_2(|v|)\geqslant L(x,v)\geqslant\theta_1(|v|)-c_0,\qquad\forall  (x,v)\in\R^n\times\R^n.
$$
\end{enumerate}
We say that a function $\theta:[0,+\infty)\to[0,+\infty)$ is superlinear if $\lim_{r\to+\infty}\frac{\theta(r)}{r}=+\infty$. In the following we call a Lagrangian $L$ a {\em Tonelli Lagrangian} if it is  $C^2$ and satisfies ({\bf{L1}}) and ({\bf{L2}}).

It is not hard to check that $L$ is a Tonelli Lagrangian if and only if the associated Hamiltonian $H:\R^n\times\R^n\to\R$, $(x,p)\mapsto \sup_{v\in \R^n}\{\langle p, v\rangle-L(x,v)\}$, is of class $C^2$ and satisfies
\begin{enumerate}[(\bf{H}1)]
\item {\it Convexity}: the Hessian $\frac{\partial^2H}{\partial p^2}(x,p)$ is positive definite for all $(x,p)\in \R^n\times \R^n$.
\item {\it Superlinearity}: there exist two nondecreasing superlinear  function $\overline{\theta}_1,\overline{\theta}_2:[0,+\infty)\to[0,+\infty)$ and a constant $c'_0>0$ such that 
$$
\overline{\theta}_2(|p|)\geqslant H(x,p)\geqslant\overline{\theta}_1(|v|)-c'_0,\qquad\forall (x,p)\in\R^n\times\R^n.
$$
\end{enumerate}
A Hamiltonian is called a Tonelli Hamiltonian if it is $C^2$ and satisfies ({\bf{H1}}) and ({\bf{H2}}). 
In fact, The collection of the conditions (H1)-(H2) above is exactly Fathi-Maderna's conditions in \cite{Fathi-Maderna}.

Let $\Omega\subset\R^n$ be a domain. We consider the following Dirichlet-type Hamilton-Jacobi equation
\begin{equation}\label{eq:HJ_BP}\tag{HJ$_g$}
	\begin{cases}
	H(x,Du)=0 \quad \text{in}\ \Omega, \\
	u\big|_{\partial\Omega}=g,
	\end{cases}
\end{equation}
where $g$ is a given continuous function on $\partial\Omega$.

\subsection{Relative fundamental solutions and Ma\~n\'e's potentials}

The main object of this paper is to study the propagation of singularities of a viscosity solution $u$ of \eqref{eq:HJ_BP}, especially the global singular dynamics governed by generalized characteristics. Therefore, an intrinsic representation formula for the solution of the problem \eqref{eq:HJ_BP} is necessary and the methods developed in \cite{Cannarsa-Cheng3} can be applied. 

In the literature, such representation formulae have already been obtained using a PDE approach (see, for instance, \cite{Ishii-Mitake}) or, in the context of control theory, as a way to investigate the value function of optimal {\em exit time} problems (\cite{Bardi-Capuzzo-Dolcetta}, \cite{Cannarsa-Sinestrari}).
Here, we are mainly interested in the interpretation of such formulae from the point of view of weak KAM theory.

In order to give an intrinsic representation formula, we need to introduce the (relative) fundamental solution. For any $x,y\in\overline{\Omega}$ and any $s<t$, we define
\begin{align*}
	\Gamma^{s,t}_{x,y}(\overline{\Omega})=\{\xi\in W^{1,1}([s,t];\R^n):\xi(\tau)\in\overline{\Omega}\ \text{for all}\ \tau\in[s,t],\ \xi(s)=x,\ \xi(t)=y\},
\end{align*}
which will be denoted by $\Gamma^{s,t}_{x,y}$ if $\Omega=\R^n$. The value function of  the problem
\begin{equation}\label{eq:CV}\tag{CV}
	\inf_{\xi\in\Gamma^{0,t}_{x,y}(\overline{\Omega})}\int^t_0L(\xi(s),\dot{\xi}(s))\ ds=\inf_{\eta\in\Gamma^{-t,0}_{x,y}(\overline{\Omega})}\int^0_{-t}L(\eta(s),\dot{\eta}(s))\ ds,
\end{equation}
denoted by $A^{\Omega}_t(x,y),$ is called the {\em relative fundamental solution} of the associated Hamilton-Jacobi equation $w_t(x,t)+H(x,w_x(x,t))=0$. A solution of (\ref{eq:CV}) is called a minimizer of $A^{\Omega}_t(x,y)$.

The following proposition is an existence and regularity result for minimizers of problem \eqref{eq:CV}. 
\begin{Pro}[Tonelli theorem \cite{Cannarsa-Sinestrari}]\label{Tonelli_straint}
	For any $t>0$ and any $x,y\in\overline{\Omega}$, problem \eqref{eq:CV} admits a solution $\xi\in\Gamma^{-t,0}_{x,y}(\overline{\Omega})$ which is Lipschitz continuous on $[-t,0\,]$. Moreover, if $-t\leqslant t_1<t_2\leqslant 0$ and $\xi(s)\in\Omega$ for $s\in(t_1,t_2)$, then $\xi\vert_{(t_1,t_2)}$ is of $C^2$ class. Furthermore, the pair $(\xi(s),p(s))$, where $p(s)=\frac{\partial L}{\partial v}(\xi(s),\dot{\xi}(s))$ is called the dual arc associated with $\xi$, satisfies the Hamiltonian system
	\begin{equation}\label{eq:Ham_ODE}
		\dot{\xi}(s)=\frac{\partial H}{\partial p}(\xi(s),p(s)),\quad\dot{p}(s)=-\frac{\partial H}{\partial x}(\xi(s),p(s)),\qquad s\in(t_1,t_2).
	\end{equation}
\end{Pro}

\begin{defn}[Extremal curve]
	Let $\gamma:[a,b]\to \R^n$ be a curve of class $C^2$ with $a<b$. Then $\gamma$ is called an extremal curve if it satisfies the Euler-Lagrange equation
	$$
	\frac{d}{dt}\frac{\partial L}{\partial v}(\gamma,\dot{\gamma})=\frac{\partial L}{\partial x}(\gamma,\dot{\gamma}).
	$$
\end{defn}

In order to give a Lipschitz estimate for the minimizers of \eqref{eq:CV}, one key point is to obtain an upper bound for $A^{\Omega}_t(x,y)$ in terms of $|x-y|/t$. In some special case, e.g., when $\Omega$ is convex, such an upper estimate can be derived by using the geodesic segment connecting $x$ and $y$. 

\begin{defn}[$C$-quasiconvex domain]\label{x_C_reachable}
	For any fixed $x\in\overline{\Omega}$ and constant $C>0$, we say that $y\in\overline{\Omega}$ is {\em $(x,C)$-reachable (in $\Omega$)}, if there exists a curve $\gamma\in \Gamma^{0,t(x,y)}_{x,y}(\overline{\Omega})$ for some $t(x,y)>0$ with $|\dot{\gamma}|=1$ a.e. on $[0,t(x,y)]$ and $t(x,y)\leqslant C|x-y|$. The set of all $(x,C)$-reachable  points is denoted by $\mathbf{R}_C(x,\Omega)$. We say that $\Omega$ is $C$-quasiconvex, if $\mathbf{R}_C(x,\Omega)=\Omega$ for all $x\in\overline{\Omega}$.
 \end{defn}

\begin{defn}[Lipschitz domain]\label{Lip domain}
A domain $\Omega\subset\R^n$ is called a Lipschitz domain, if $\partial\Omega$ is locally Lipschitz, i.e., can be locally represented as the graph of a Lipschitz function defined on some open ball of $\R^{n-1}$.

\end{defn}

\begin{Rem}\label{quasiconvex}
	It is a fact that $\Omega$ is $1$-quasiconvex only if it is convex. It is not difficult to show that $\Omega$ is $C$-quasiconvex for some $C>0$ if $\Omega$ is a bounded Lipschitz domain (see, for instance, Sections 2.5.1 and 2.5.2 in \cite{BB}). For more on length spaces and $C$-quasiconvex domains see \cite{Gromov}.
\end{Rem}

\begin{Lem}[{\it A priori} Lipschitz estimate for minimizers]\label{eq:Lip_main}
Let $\Omega$ be a $C$-quasiconvex domain for some $C>0$.
Then there exists a nondecreasing superlinear function $\kappa:[0,+\infty)\to[0,+\infty)$ such that, for any $x,y\in\overline{\Omega}$, any $t>0$ and any minimizer $\xi\in\Gamma^{-t,0}_{x,y}(\overline{\Omega})$ of $A^{\Omega}_t(x,y)$, we have
\begin{align*}
	\operatorname*{ess\ sup}_{s\in[-t,0]}|\dot{\xi}(s)|\leqslant\kappa(C|x-y|/t).
\end{align*}
 Moreover, if $\Omega$ is bounded in addition, then we get
\begin{align*}
	\operatorname*{ess\ sup}_{s\in[-t,0]}|\dot{\xi}(s)|\leqslant\kappa(CD/t),
\end{align*}
where $D>0$ denotes the diameter of $\Omega$.
\end{Lem}

\begin{proof}
Fix $x,y\in\overline{\Omega}$ and $t>0$. Since $\Omega$ is $C$-quasiconvex, there is a curve $\gamma\in\Gamma^{0,t(x,y)}_{x,y}(\overline{\Omega})$ satisfying with $|\dot{\gamma}|=1$ a.e. on $[0,t(x,y)]$ and $t(x,y)\leqslant C|x-y|$.
	Define a curve $\eta\in\Gamma^{-t,0}_{y,x}(\overline{\Omega})$ by
	$$
	\eta(s)=\gamma\left(\frac{t(x,y)}t(s+t)\right),\quad s\in[-t,0].
	$$
	Then
	\begin{equation}\label{eq:upper_bound1}
	A^{\Omega}_t(y,x)\leqslant\int^0_{-t}L(\eta(s),\dot{\eta}(s))\ ds\leqslant t\theta_2\left(\frac{t(x,y)}t\right)\leqslant t\theta_2\left(\frac{C|x-y|}t\right).
	\end{equation}
	On the other hand, we have 
	\begin{equation}\label{eq:lower_bound1}
	\begin{split}
		A^{\Omega}_t(x,y)\geqslant&\,\int^0_{-t}\theta_1(|\dot{\xi}(s)|)\ ds-c_0t\geqslant\int^0_{-t}|\dot{\xi}(s)|\ ds-(\theta_1^*(1)+c_0)t.
	\end{split}
	\end{equation}
	The combination of \eqref{eq:upper_bound1} and \eqref{eq:lower_bound1} leads to
	$$
	\int^0_{-t}|\dot{\xi}(s)|\ ds\leqslant t\left\{\theta_2\left(\frac{C|x-y|}t\right)+\theta_1^*(1)+c_0\right\},
	$$
	which implies that 
	$$
	|\xi(s)-x|\leqslant\int^0_{-t}|\dot{\xi}(s)|\ ds\leqslant t\left\{\theta_2\left(\frac{C|x-y|}t\right)+\theta_1^*(1)+c_0\right\},
	$$
	and
	$$
	\operatorname*{ess\ inf}_{s\in[0,t]}|\dot{\xi}(s)|\leqslant\theta_2\left(\frac{C|x-y|}t\right)+\theta_1^*(1)+c_0.
	$$
	The rest of the proof is standard, see, for instance, \cite{Ambrosio-Ascenzi-Buttazzo}, \cite{Dal-Maso-Frankowska} or Proposition A.1 in \cite{Cannarsa-Cheng3}.
\end{proof}

\begin{Pro}[Regularity properties of relative fundamental solutions]\label{C11_A_t}For the regularity of $A^{\Omega}_t(x,y)$, we have
\begin{enumerate}[\rm (a)]
  \item for any $x\in\Omega$ and any $t>0$, $y\mapsto A^{\Omega}_t(x,y)$ is locally semiconcave in $\Omega$;
  \item for any $x\in\Omega$ and any $\lambda>0$, there exists $t_\lambda>0$ such that, for any $0<t\leqslant\min\{t_{\lambda},d_{\partial\Omega}(x)/\kappa(\lambda)\}$, the function $y\mapsto A^{\Omega}_t(x,y)$ is  uniformly convex on $B(x,\lambda t)\subset\Omega$ with a constant $C(\lambda)/t$, where $\kappa$ is the function obtained in Lemma \ref{eq:Lip_main};
  \item for any $x\in\Omega$, the functions $y\mapsto A^{\Omega}_t(x,y)$ and $y\mapsto A^{\Omega}_t(y,x)$ are of class $C^{1,1}_{\text{loc}}(B(x,\lambda t))$ if $0<t\leqslant\min\{t_{\lambda},d_{\partial\Omega}(x)/\kappa(\lambda)\}$. Moreover, for all $y\in B(x,\lambda t)$, we have
\begin{equation}\label{eq:diff_A_t}
	D_yA^{\Omega}_t(x,y)=\frac{\partial L}{\partial v}(\xi(t),\dot{\xi}(t)),\quad 
		D_xA^{\Omega}_t(x,y)=-\frac{\partial L}{\partial v}(\xi(0),\dot{\xi}(0)),
\end{equation}
where $\xi\in\Gamma^{0,t}_{x,y}(\overline{B(x,\kappa(\lambda)t)})$ is the unique minimizer for $A^{\Omega}_t(x,y)$.
\end{enumerate}
\end{Pro}

\begin{proof}[Sketch of proof]
	Let $x\in\Omega$, $R=d_{\partial\Omega}(x)$ and let $\kappa$ be the function given by Lemma \ref{eq:Lip_main} (relative to $B(x,R)$ which is 1-quasiconvex). Then, for any $\lambda>0$ let $t_{\lambda}>0$ be such that $\max\{\lambda,\kappa(\lambda)\}\leqslant \frac {R}{t_{\lambda}}$. By Lemma \ref{eq:Lip_main}, for any $t\in(0,t_{\lambda}]$, any $y\in B(x,\lambda t)$ and any $\xi\in\Gamma^{0,t}_{x,y}(\overline{\Omega})$ minimizing $A^{\Omega}_t(x,y)$, we have
	\begin{align*}
		\xi(s)\in B(x,R)\subset\Omega,\quad s\in[0,t].
	\end{align*}
	Therefore, 
	$$
	A^{\Omega}_t(x,y)=A_t(x,y),\quad \forall t\in(0,t_{\lambda}],\ \forall y\in B(x,\lambda t),
	$$
	where $A_t(x,y)$ denotes the relative fundamental solution for $\Omega=\R^n$.
	Then, all the statements follow from the regularity results in \cite[Appendix A]{Cannarsa-Cheng3}.
\end{proof}

\begin{defn}[Relative Ma\~n\'e's potential and critical value]
	The function $\Phi^{\Omega}_L:\overline{\Omega}\times\overline{\Omega}\to\R\cup\{-\infty\}$ defined as
\begin{equation}
	\Phi^{\Omega}_L(x,y)=\inf_{t>0}A_t^{\Omega}(x,y),\quad \forall x,y\in\overline{\Omega},
\end{equation}
is called {\em Ma\~n\'e's potential} associated with $L$  relative to $\overline{\Omega}$ or {\em relative Ma\~n\'e's potential} for short. We call the value defined by
\begin{equation}\label{eq:mane}
	c_{\Omega}(L)=-\inf_{t>0,x\in\overline{\Omega}}\frac 1tA^{\Omega}_t(x,x),
\end{equation}
{\em Ma\~n\'e's critical value} of $L$ relative to $\overline{\Omega}$ or {\em relative Ma\~n\'e's critical value} for short.
\end{defn}

We collect some elementary properties of the relative Ma\~n\'e's potential  and relative Ma\~n\'e's critical value here. See Appendix~\ref{sse:a2} for the proofs of Lemma \ref{poten_proper} and Lemma \ref{eq:equiv_Phi_0}.

\begin{Lem}\label{poten_proper}
The relative Ma\~n\'e's potential $\Phi^{\Omega}_L$ has the following properties:
\begin{enumerate}[\rm (1)]
	\item $\Phi^{\Omega}_L(x,z)\leqslant \Phi^{\Omega}_L(x,y)+\Phi^{\Omega}_L(y,z)$ for all $x$, $y$, $z\in\overline{\Omega}$;
	\item if $c_{\Omega}(L)\leqslant 0$, then $\Phi^{\Omega}_L(x,x)=0$ for all $x\in\overline{\Omega}$;
	\item if $c_{\Omega}(L)\leqslant 0$ and $\Omega$ is $C$-quasiconvex for some constant $C>0$, then
	\begin{enumerate}[\rm (i)]
	 \item $
	|\Phi^{\Omega}_L(x,y)|\leqslant \theta_2(1)C|x-y|,\quad\forall x,y\in\overline{\Omega};
	$
	 \item $\Phi^{\Omega}_L(\cdot,\cdot)$ is Lipschitz with a Lipschitz constant $\theta_2(1)C$.
	 \end{enumerate}	
\end{enumerate}	
\end{Lem}

\begin{Lem}\label{eq:equiv_Phi_0}
The following statements are equivalent:
\begin{enumerate}[\rm (1)]
	\item $\Phi^{\Omega}_L(x,y)>-\infty$ for any $x,y\in\overline{\Omega}$;
	\item there exist $x,y\in\overline{\Omega}$ such that $\Phi^{\Omega}_L(x,y)>-\infty$;
	\item $c_{\Omega}(L)\leqslant 0$.
\end{enumerate}	
\end{Lem}

\subsection{Semiconcave functions}

Let $S$ be a nonempty subset of $\R^n$.

\begin{defn}[Semiconcave functions]\label{defn_semiconcave}
We recall that a function $u:S\to\R$ is said to be {\em semiconcave} (with linear modulus) if there exists a constant $C>0$ such that
\begin{equation}\label{eq:SCC}
\lambda u(x)+(1-\lambda)u(y)-u(\lambda x+(1-\lambda)y)\leqslant\frac C2\lambda(1-\lambda)|x-y|^2
\end{equation}
for any $x,y\in S$, such that the segment $[x,y]$ is contained in $S$ and any $\lambda\in[0,1]$.  Any constant $C$ that satisfies the above inequality  is called a {\em constant of semiconcavity} for $u$ in $S$. A function $u:S\to\R$ is said to be {\em semiconvex} if $-u$ is semiconcave. 
\end{defn}

If $S$ is a convex subset of $\R^n$, then $u$ is semiconcave with constant $C$ if \eqref{eq:SCC} holds for $x,y\in S$ and all $\lambda\in[0,1]$. Let $S$ be an open subset of $\R^n$. A function $u:S\rightarrow\R$ is said to be {\em locally semiconcave} (resp. {\em locally semiconvex}) if for each $x\in S$ there exists an open ball $B(x,r)\subset S$ such that $u$ is a semiconcave (resp. semiconvex) function on $B(x,r)$. We say $u:\overline{S}\to\R$ is semiconcave up to the boundary if $u$ is a semiconcave function on $\overline{S}$ as in the Definition \ref{defn_semiconcave}.

The following result shows that in order to prove a given function is semiconcave with linear modulus, it is sufficient to show \eqref{eq:SCC} for the midpoint of any segment. One can find the proof in \cite{Cannarsa-Sinestrari}.
\begin{Pro}
Let $u:S\to\R$ be continuous. Then $u$ is semiconcave with constant $C$ if   
$$
u(x)+u(y)-2u\left(\frac{x+y}2\right)\leqslant \frac C2|x-y|^2
$$
for any $x,y$ such that the segment $[x,y]$ is contained in $S$.
\end{Pro}

Hereafter, assume $S$ is an open subset of $\R^n$. Let $u:S\subset\R^n\to\R$ be a continuous function. We recall that, for any $x\in S$, the closed convex sets
\begin{align*}
D^-u(x)&=\left\{p\in\R^n:\liminf_{y\to x}\frac{u(y)-u(x)-\langle p,y-x\rangle}{|y-x|}\geqslant 0\right\},\\
D^+u(x)&=\left\{p\in\R^n:\limsup_{y\to x}\frac{u(y)-u(x)-\langle p,y-x\rangle}{|y-x|}\leqslant 0\right\},
\end{align*}
are called the {\em (Dini) subdifferential} and {\em superdifferential} of $u$ at $x$, respectively.

Let $u:S\to\R$ be a locally Lipschitz function. We recall that a vector $p\in\R^n$ is said to be a {\em reachable} (or {\em limiting}) {\em gradient} of $u$ at $x$ if there exists a sequence $\{x_k\}\subset S\setminus\{x\}$, converging to $x$, such that $u$ is differentiable at $x_k$ for each $k\in\N$ and
$$
\lim_{k\to\infty}Du(x_k)=p.
$$
The set of all reachable gradients of $u$ at $x$ is denoted by $D^{\ast}u(x)$.

\begin{Pro}[Superdifferential of semiconcave functions \cite{Cannarsa-Sinestrari}]\label{basic_facts_of_superdifferential}
Let $u:S\subset\R^n\to\R$ be a semiconcave function and let $x\in S$. Then the following properties hold:
\begin{enumerate}[\rm {(}a{)}]
  \item $D^+u(x)$ is a nonempty compact convex set in $\R^n$ and $D^{\ast}u(x)\subset\partial D^+u(x)$, where  $\partial D^+u(x)$ denotes the topological boundary of $D^+u(x)$;
  \item the set-valued function $x\rightsquigarrow D^+u(x)$ is upper semicontinuous;
  \item $D^+u(x)\not=\varnothing$ and $D^+u(x)$ is a singleton if $D^-u(x)\not=\varnothing$ in addition;
  \item if $D^+u(x)$ is a singleton, then $u$ is differentiable at $x$. Moreover, if $D^+u(x)$ is a singleton for every point in $S$, then $u\in C^1(S)$.
\end{enumerate}
\end{Pro}

\subsection{Generalized characteristics}

A basic criterion for the propagation of singularities of  viscosity solutions to Hamilton-Jacobi equations along generalized characteristics was given in \cite{Albano-Cannarsa} (see \cite{Cannarsa-Yu,Yu} for an improved version and a simplified proof of this result). 

\begin{defn}[Generalized characteristic]
A Lipschitz arc $\mathbf{x}:[0,T]\to\Omega,\,(T>0),$ is said to be a {\em generalized characteristic} of the Hamilton-Jacobi equation \eqref{eq:HJ_BP} if $\mathbf{x}$ satisfies the differential inclusion
\begin{equation}\label{generalized_characteristics}
\dot{\mathbf{x}}(s)\in\mathrm{co}\, H_p\big(\mathbf{x}(s),D^+u(\mathbf{x}(s))\big),\quad \text{a.e.\ in}\ [0,T]\,.
\end{equation}
\end{defn}


\section{Representation formula for the solution of Dirichlet problem}

Throughout this section the following standing hypotheses $(\bf{SH})$ will be assumed without further notice:
\begin{enumerate}[({\bf{SH}}1)]
  \item $\Omega\subset \R^n$ is a bounded Lipschitz domain;
  \item $L$ is a Tonelli Lagrangian;
  \item $g:\partial\Omega\to\R$ is a function satisfying the compatibility condition
\begin{equation}\label{eq:supcritical1}
	g(x)-g(y)\leqslant\Phi^{\Omega}_L(y,x),\quad \forall x,y\in\partial\Omega;
\end{equation} 
\item the relative Ma\~n\'e's critical value of $L$  satisfies the energy condition
 \begin{equation}\label{eq:supcritical2}
	c_{\Omega}(L)\leqslant0.
\end{equation}
\end{enumerate}

Consider the following minimization problem 
\begin{equation}\label{eq:represent_formulae}\tag{CV$_g$}
	u(x):=\inf_{y\in\partial\Omega}\{g(y)+\Phi^{\Omega}_L(y,x)\},\quad x\in\overline{\Omega}.
\end{equation}

{\em From now on, $u:\overline{\Omega}\to\R$ denotes the value function of \eqref{eq:represent_formulae}}. 

\vskip.2cm

The following facts are immediate consequences of $(\bf{SH})$:
\begin{itemize}[$\bullet$]
   \item since $\Omega$ is a bounded Lipschitz domain, in view of Remark \ref{quasiconvex}, $\Omega$ is a $C$-quasiconvex domain for some constant $C>0$;
   \item since $g$ satisfies \eqref{eq:supcritical1} and $c_{\Omega}(L)\leqslant 0$, by invoking Lemma \ref{poten_proper} (3) we deduce that,  for any $x,y\in\partial\Omega$,
	\begin{align*}
		g(x)-g(y)\leqslant\Phi^{\Omega}_L(y,x)\leqslant C_1|x-y|,
	\end{align*}
	where $C_1:=\theta_2(1)C$. Thus, \eqref{eq:supcritical1} and \eqref{eq:supcritical2} together imply that $g$ is Lipschitz on $\partial\Omega$;
  \item in view of Lemma \ref{poten_proper}, $u(x)>-\infty$ for all $x\in\overline{\Omega}$;
  \item since both $\Phi^{\Omega}_L$ and $g$ are Lipschitz continuous functions, the infimum defining $u$ is attained at some point $y_x\in\overline{\Omega}$, which will be called a minimizer for \eqref{eq:represent_formulae} at $x$.
  \end{itemize}

\begin{Pro}\label{compatible_u}
	$u$ is dominated by $L$, i.e., 
	\begin{equation}\label{eq:sub_solution}
		u(x')-u(x)\leqslant\Phi^{\Omega}_L(x,x'),\quad\forall x,x'\in\overline{\Omega}.
	\end{equation}
	Moreover, $u$ is Lipschitz on $\overline{\Omega}$ and 
$u=g$ on $\partial\Omega$. 
\end{Pro}

\begin{proof}
	Let $y$ be a minimizer of \eqref{eq:represent_formulae} at $x$. Then,
	by the definition of $u$ and Lemma \ref{poten_proper} (1), we get
	\begin{align*}
		u(x')-u(x)
		\leqslant&\Phi^{\Omega}_L(y,x')-\Phi^{\Omega}_L(y,x)\leqslant\Phi^{\Omega}_L(x,x'),\quad \forall x'\in\overline{\Omega}. 
	\end{align*}
	Now, let $x\in\partial\Omega$. By Lemma \ref{poten_proper} (2) and (\ref{eq:supcritical1}), we have
	\begin{align*}
		u(x)\leqslant g(x)+\Phi^{\Omega}_L(x,x)=g(x)\leqslant\inf_{y\in\partial\Omega}\{g(y)+\Phi^{\Omega}_L(y,x)\}=u(x).
	\end{align*}
	So, $u\vert_{\partial\Omega}=g$.  Lemma \ref{poten_proper} (3) and (\ref{eq:sub_solution}) imply that $u$ is Lipschitz on $\overline{\Omega}$.
\end{proof}

\subsection{Exit time problem} Problem \eqref{eq:represent_formulae} is closely related to the so called {\em exit time problem} in optimal control, and the readers can refer to Chapter IV of \cite{Bardi-Capuzzo-Dolcetta} or Chapter 8 of \cite{Cannarsa-Sinestrari} for more on this topic. In order to adapt such a problem  to the context of weak KAM theory, we will slightly modify the standard terminology. Moreover, in order to give a more precise formulation of the Dirichlet problem \eqref{eq:HJ_BP}, we need a Lipschitz estimate for the associated minimal curves in \eqref{eq:represent_formulae}. Now, we will show that \eqref{eq:represent_formulae} is indeed an exit time problem. 

\begin{defn}[Calibrated curve]
	 We say that a curve $\xi\in\Gamma^{a,b}_{x,y}(\overline{\Omega})$ with $a<b$ is a {\em $(u,L,\Omega)$-calibrated curve}, or $u$-calibrated curve for short, if 
	$$
	u(\xi(b))-u(\xi(a))=\int^b_aL(\xi,\dot{\xi})\ ds.
	$$
	A curve $\xi:(-\infty,0\,]\to\R^n$ with $\xi(s)\in\overline{\Omega}$ for all $s\in(-\infty,0\,]$, is called a $u$-calibrated curve if it is a $u$-calibrated curve on each compact sub-interval of $(-\infty,0\,]$.
\end{defn}

\begin{Pro}\label{pro_calibrated}
Let $x,y\in\overline{\Omega}$. Then
\begin{enumerate}[\rm (a)]
  \item if $\xi\in\Gamma^{a,b}_{x,y}(\overline{\Omega})$, $-\infty<a<b<+\infty$, is a $u$-calibrated curve, then the restriction of $\xi$ to any sub-interval of $[a,b]$ is still a $u$-calibrated curve;
  \item let  $M(x)$ be the set of all minimizers of \eqref{eq:represent_formulae} at $x$. Then
  \begin{enumerate}[\rm (i)]
  \item if there exists $y\in M(x)$ such that $A^{\Omega}_T(y,x)=\Phi^{\Omega}_L(y,x)$ for some $T\in(0,+\infty)$, then any minimizer $\xi$ of $A^{\Omega}_T(y,x)$ is $u$-calibrated and, for any $t\in[0,T]$ with $\xi(-t)\in\partial\Omega$, we have 
    $$
  u(x)=g(\xi(-t))+A^{\Omega}_T(\xi(-t),x);
  $$
  \item if $A^{\Omega}_T(y,x)>\Phi^{\Omega}_L(y,x)$ for all $T\in(0,+\infty)$ and all $y\in M(x)$, then there exists a $u$-calibrated curve $\xi:(-\infty,0\,]\to\R^n$, with $\xi(0)=x$ and $\xi(s)\in\overline{\Omega}$ for all $s\in(-\infty,0\,]$,  such that $\xi(-t)\not\in\partial\Omega$ for all $t>0$.
  \end{enumerate}
\end{enumerate}
\end{Pro}

\begin{proof}
Suppose $\xi\in\Gamma^{a,b}_{x,y}(\overline{\Omega})$ is a $u$-calibrated curve. Let $a\leqslant c<d\leqslant b$. Then, we have
\begin{align*}
	u(\xi(b))-u(\xi(d))\leqslant&\int^b_dL(\xi,\dot{\xi})\ ds,\\
	u(\xi(d))-u(\xi(c))\leqslant&\int^d_cL(\xi,\dot{\xi})\ ds,\\
	u(\xi(c))-u(\xi(a))\leqslant&\int^c_aL(\xi,\dot{\xi})\ ds,
\end{align*}
by Proposition \ref{compatible_u}. On the other hand, the sum of the left side of the three inequalities above equals to that of the right side since $\xi$ is $u$-calibrated. Thus each inequality should be an equality. Thus, $\xi$ restricted to $[c,d]$, is also $u$-calibrated. This completes the proof of (a).

Now we turn to the proof of (b). Let $y$ be a minimizer of \eqref{eq:represent_formulae} at $x$. If $A^{\Omega}_T(y,x)=\Phi^{\Omega}_L(y,x)$ for some $T\in(0,+\infty)$, then, recalling Proposition \ref{compatible_u}, we have that
$$
u(x)=g(y)+A^{\Omega}_T(y,x)=u(y)+\int^0_{-T}L(\xi,\dot{\xi})\ ds,
$$
where $\xi\in\Gamma^{-T,0}_{y,x}(\overline{\Omega})$ is a minimizer of $A^{\Omega}_T(y,x)$. Thus $\xi$ is $u$-calibrated and, if $\xi(-t)\in\partial\Omega$ for some $t\in[0,T]$, we have that
$$
u(x)=u(\xi(-t))+\int^0_{-t}L(\xi,\dot{\xi})\ ds\geqslant g(\xi(-t))+\Phi^{\Omega}_L(\xi(-t),x)\geqslant u(x),
$$
where we have used (a), the fact that $u\vert_{\partial\Omega}=g$ and the definition of $u$. So far, we have proved (i).

Finally, suppose $A^{\Omega}_T(y,x)>\Phi^{\Omega}_L(y,x)$ for all $T\in(0,+\infty)$ and all $y\in M(x)$. Then there exists a sequence $(T_k,y_k)\in(0,+\infty)\times\partial\Omega$, with $\lim_{k\to\infty}T_k=+\infty$, and minimizers $\xi_k\in\Gamma^{-T_k,0}_{y_k,x}(\overline{\Omega})$ of $A^\Omega_{T_k}(y_k,x)$ such that
\begin{equation}\label{eq:calibrated_1}
	u(x)>u(\xi_k(-T_k))+\int^0_{-T_k}L(\xi_k,\dot{\xi}_k)\ ds-\frac 1k.
\end{equation}
Now, fix any $t>0$ and choose $T_k>t$. Recalling \eqref{eq:sub_solution}, we conclude that
\begin{align}
	u(x)-u(\xi_k(-t))\leqslant&\int^0_{-t}L(\xi_k,\dot{\xi}_k)\ ds,\label{eq:calibrated_2}\\
	u(\xi_k(-t))-u(\xi_k(-T_k))\leqslant&\int^{-t}_{-T_k}L(\xi_k,\dot{\xi}_k)\ ds.\label{eq:calibrated_3}
\end{align}
The combination of \eqref{eq:calibrated_1} and \eqref{eq:calibrated_3} leads to
\begin{equation}\label{eq:calibrated_4}
	\begin{split}
		u(x)\geqslant&\,u(\xi_k(-t))-\int^{-t}_{-T_k}L(\xi_k,\dot{\xi}_k)\ ds+\int^0_{-T_k}L(\xi_k,\dot{\xi}_k)\ ds-\frac 1k\\
	=&\,u(\xi_k(-t))+\int^0_{-t}L(\xi_k,\dot{\xi}_k)\ ds-\frac 1k.
	\end{split}
\end{equation}

 By Lemma \ref{eq:Lip_main}, $\{\xi_k\}$ is equi-Lipschitz for $k$ large enough. Then, combining \eqref{eq:calibrated_2} and \eqref{eq:calibrated_4}, invoking the Arzel\`a-Ascoli theorem, and taking a subsequence if necessary, we conclude that there exists a Lipschitz curve $\xi:(-\infty,0\,]\to\R^n$, with $\xi(0)=x$ and $\xi(s)\in\overline{\Omega}$ for all $s\in(-\infty,0\,]$, such that
\begin{gather}
	u(x)=u(\xi(-t))+\int^0_{-t}L(\xi,\dot{\xi})\ ds,\quad t>0.
\end{gather}
This proves that $\xi$ is $u$-calibrated. If $\xi(-t)\in\partial\Omega$ for some $t>0$, then
\begin{align*}
	u(x)=u(\xi(-t))+\int^0_{-t}L(\xi,\dot{\xi})\ ds=g(\xi(-t))+\int^0_{-t}L(\xi,\dot{\xi})\ ds\geqslant u(x).
\end{align*}
Then we have that $\xi(-t)$ is a minimizer of \eqref{eq:represent_formulae} and $A^{\Omega}_t(\xi(-t),x)=\Phi^{\Omega}_L(\xi(-t),x)$ which contradicts our assumption. This completes the proof of (ii).
\end{proof}

From now on we will impose a stronger condition than ({\bf{SH}}4) on the relative Ma\~n\'e's critical value of $L$: 
\medskip

\noindent({\bf{SH}}4')  \ \ \ \qquad \qquad \qquad \ \ \ \ \ \ \qquad
	$c_{\Omega}(L)<0$.

\begin{Lem}\label{T_bound}
	For any $x\in\overline{\Omega}$ and any minimizer $y^*\neq x$ of \eqref{eq:represent_formulae} at $x$, we have that $A^{\Omega}_{T}(y^*,x)=\Phi^{\Omega}_L(y^*,x)$ for some $T>0$. Moreover, 
	$$
	T\leqslant \frac{2C_1}{-c_{\Omega}(L)}|y^*-x|,
	$$
	where $C_1:=\theta_2(1)C$ that has been defined at the beginning of this section.
\end{Lem}

\begin{proof}
	Let $L_1=L+c_{\Omega}(L)$. Then $c_{\Omega}(L_1)=0$.  In view of Lemma \ref{poten_proper} (3), for any $t>0$ and $x,y\in\overline{\Omega}$ we obtain
	\begin{align}\label{eq:pot}
		A^{\Omega}_t(y,x)+c_{\Omega}(L)t\geqslant\Phi^{\Omega}_{L_1}(y,x)\geqslant -C_1|x-y|.
	\end{align}
	
	Now, fix any $x\in\overline{\Omega}$ and let $y^*\neq x$ be a minimizer of \eqref{eq:represent_formulae} at $x$. We claim that $\Phi^{\Omega}_L(y^*,x)=A^{\Omega}_T(y^*,x)$  for some $T>0$. Otherwise, by Proposition \ref{compatible_u} and Lemma \ref{poten_proper}, there would exist a sequence $T_k\to+\infty$ such that
	\begin{equation}\label{eq:pot2}
		A^{\Omega}_{T_k}(y^*,x)\leqslant u(x)-u(y^*)+\frac 1k\leqslant\Phi^{\Omega}_L(y^*,x)+\frac 1k\leqslant C_1|y^*-x|+\frac 1k.
	\end{equation}
	The combination of \eqref{eq:pot} and\eqref{eq:pot2} leads to 
	$$
	-c_{\Omega}(L)T_k\leqslant 2C_1|y^*-x|+\frac 1k,
	$$
	which is impossible.
	
	Therefore,  again by Lemma \ref{poten_proper} (3) we have that 
	$$A^{\Omega}_T(y^*,x)=\Phi^{\Omega}_L(y^*,x)\leqslant C_1|y^*-x|.$$ 
	Thus, owing to \eqref{eq:pot},
	$$
	T\leqslant\frac{2C_1}{-c_{\Omega}(L)}|y^*-x|.
	$$
This completes the proof.
\end{proof}


\begin{defn}[Exit time function]\label{exit_time_function}
We define the {\em exit time function} $T:\overline{\Omega}\to(0,+\infty]$ by

$$
T(x)=\inf_{\xi}T_{\xi}(x),\quad \forall x\in\overline{\Omega},
$$
where the infimum is taken among $u$-calibrated curves $\xi$ defined on $[-T,0\,]$ or $(-\infty,0\,]$ with $\xi(0)=x$, and 
$$
T_{\xi}(x)=\inf\{s\geqslant 0: \xi(-s)\in\partial\Omega \}.
$$
\end{defn}
By definition it is easy to see that $T(x)=0$ for all $x\in\partial \Omega$. From Lemma \ref{T_bound}, there is a constant $B>0$ such that 
\begin{align}\label{T-up}
T(x)\leqslant B,\quad \forall x\in\overline{\Omega}.
\end{align}
\begin{Lem}
	Exit time function $x\mapsto T(x)$ is lower semicontinuous on $\overline{\Omega}$.
\end{Lem}
\begin{proof}
	For any $x_0\in\overline{\Omega}$, let $T_0=\liminf_{x\to x_0}T(x)$. It suffices to show: for any $\{x_k\}_{k\in \N}\subset \overline{\Omega}$ such that $x_k\to x_0$ and $T(x_k)\to T_0$ as $k\to+\infty$, we have
	\begin{align}\label{lsc}
	T(x_0)\leqslant T_0.
	\end{align}
	Since $T(x)\equiv 0$ on $\partial \Omega$, then \eqref{lsc} is true if $x_0\in\partial \Omega$.
	If $x_0\in\Omega$, we suppose by contradiction that
	$T_0<T(x_0)$. Recall $T(x_k)=\inf_{\xi}T_{\xi}(x_k)$.
	Then for each $k\in\N$, there is $\xi_k:(-\infty,0]\to\overline{\Omega}$ which is $u$-calibrated with $\xi_k(0)=x_k$, such that
	$$
	T_{\xi_k}(x_k)<T(x_k)+\frac 1k.
	$$
	Let $T_k:=T_{\xi_k}(x_k)$. Then by the definition of $T_{\xi_k}(x_k)$, we have $\xi_k(-T_k)\in\partial \Omega$. Thus, for any $\varepsilon>0$, we have
	$$
	T(x_k)\leqslant T_k<T(x_k)+\frac 1k\leqslant T_0+\varepsilon
	$$
	for $k$ large enough. Since $\varepsilon>0$ is arbitrary, we get 
	$$
	T_0=\lim_{k\to +\infty}T_k.
	$$
	Recall that for each $k\in\N$, $\xi_k$ is a $u$-calibrated curve. Then,
	$$
	H(\xi_k(s),p_k(s))=0, \quad p_k(s)=\frac{\partial L}{\partial v}(\xi_k(s),\dot{\xi}_k(s)), \quad \forall s\in (-\infty,0].
	$$
	Since $\overline{\Omega}$ is compact, $H(x,p)$ is superlinear in $p$, then $\{|\dot{\xi}_k(0)|\}_{k\in \N}$ is bounded from above. Without loss of generality, suppose $(x_k,\dot{\xi}_k(0))\to (x_0,v_0)$ as $k\to +\infty$. Denote by $(\xi_\infty,\dot{\xi}_\infty)$ the solution of the Euler-Lagrange equation with $(x_0,v_0)$ as the initial condition. 
	By the classical theory of ordinary differential equations, $\xi_k$ converges to $\xi_\infty$ uniformly on $[-T_0-\varepsilon,0]$, and $\xi_\infty$ is still a $u$-calibrated curve. Since 
	$$
	|\xi_k(-T_k)-\xi_\infty(-T_0)|\leqslant |\xi_k(-T_k)-\xi_\infty(-T_k)|+|\xi_\infty(-T_k)-\xi_\infty(-T_0)|,
	$$
	then we have
	$\lim_{k\to+\infty}\xi_k(-T_k)=\xi_\infty(-T_0)$, which implies that $\xi_\infty(-T_0)\in\partial \Omega$. Therefore, we deduce that $T(x_0)\leqslant T_0$,
	a contradiction.
\end{proof}

The following result insures that the function defined in \eqref{eq:represent_formulae} is indeed the value function of an optimal exit time problem. 
\begin{Cor}\label{exit_time_prob}
For every $x\in\Omega$,  there exists $y\in\partial\Omega$ such that
$$
u(x)=g(y)+A^{\Omega}_{T(x)}(y,x).
$$	
\end{Cor}
\noindent
The proof consists of a direct application of Lemma \ref{T_bound} and Proposition \ref{pro_calibrated} (b) (i).

\begin{Lem}\label{upper_bound}
Let $x\in\Omega$. If $y^*\in\partial\Omega$ is a minimizer of \eqref{eq:represent_formulae} and there exists $T>0$ such that $u(x)=g(y^*)+A^{\Omega}_T(y^*,x)$, then we have 
\begin{equation}\label{eq:Tx_dx}
	|y^*-x|\leqslant C_2T,
\end{equation}
where $C_2=\theta_1^*(C\theta_2(\kappa(1))+1)+c_0$, $C$ is the constant for which $\Omega$ is a $C$-quasiconvex domain, and $\kappa$ is the function obtained in Lemma \ref{eq:Lip_main}.
In particular, $d_{\partial\Omega}(x)\leqslant C_2T(x)$. Furthermore, 
if $\xi^*$ is a minimizer of $A^{\Omega}_T(y^*,x)$, then
\begin{align*}
	\operatorname*{ess\ sup}_{s\in[-T,0]}|\dot{\xi^*}(s)|\leqslant\kappa(CC_2).
\end{align*}
\end{Lem}

\begin{proof}
	Fix $x\in\Omega$. Let $y^*\in\partial\Omega$ be a minimizer of \eqref{eq:represent_formulae} and let $T>0$ be a constant such that $u(x)=g(y^*)+A^{\Omega}_T(y^*,x)$. Then, for any $t>0$ and any $y\in\partial\Omega$, we get
\begin{equation*}\label{eq:2}
	\begin{split}
		0\leqslant&\, g(y)+A^{\Omega}_t(y,x)-(g(y^*)+A^{\Omega}_{T}(y^*,x))\\
	=&\, g(y)-g(y^*)+A^{\Omega}_t(y,x)-A^{\Omega}_{T}(y^*,x).
	\end{split}
\end{equation*}
By taking $y=y^*$, for any $t>0$ and any minimizer $\eta\in\Gamma^{-t,0}_{y^*,x}(\overline{\Omega})$ of $A^\Omega_t(y^*,x)$, we have
\begin{align*}
	A^{\Omega}_{T}(y^*,x)\leqslant \int^0_{-t}L(\eta(s),\dot{\eta}(s))\ ds\leqslant \int^0_{-t}\theta_2(|\dot{\eta}(s)|)\ ds.
\end{align*}
By Lemma \ref{eq:Lip_main}, we have 

\begin{align*}
	\operatorname*{ess\ sup}_{s\in[0,t]}|\dot{\eta}(s)|\leqslant\kappa(C|x-y^*|/t).
\end{align*}
Thus, we have
 \begin{align*}
	A^{\Omega}_{T}(y^*,x)\leqslant \int^0_{-t}\theta_2(\kappa(C|x-y^*|/t))\ ds\leqslant t\theta_2\bigg(\kappa\big(\frac{C|x-y^*|}t\big)\bigg).
\end{align*}
Taking $t=C|x-y^*|$, then
\begin{equation}\label{eq:upper_bound}
A^{\Omega}_{T}(y^*,x)\leqslant C\theta_2(\kappa(1))|x-y^*|.
\end{equation}
On the other hand, condition (L2) also implies that, for each $k>0$,
\begin{equation}\label{eq:lower_bound2}
	A^{\Omega}_{T}(y^*,x)\geqslant k|y^*-x|-(\theta^*_1(k)+c_0)T.
\end{equation}
Let $k=C\theta_2(\kappa(1))+1$.
Combining \eqref{eq:upper_bound} and \eqref{eq:lower_bound2}, we have
$$
|x-y^*|\leqslant C_2T,
$$
where $C_2=\theta_1^*(C\theta_2(\kappa(1))+1)+c_0$.
In particular, by Corollary \ref{exit_time_prob} and  \eqref{eq:Tx_dx}, we get $d_{\partial\Omega}(x)\leqslant C_2 T(x)$. The last conclusion of the lemma is a direct consequence of Lemma \ref{eq:Lip_main} and \eqref{eq:Tx_dx}.
\end{proof}

\subsection{Local semiconcavity of $u$}

In this part, we begin with  a local semiconcavity estimate.

For any $\rho>0$, let $\Omega_\rho=\{x\in\Omega: d_{\partial \Omega}(x)\geqslant \rho\}$. Since $T(x)$ is lower semicontinuous on $\overline{\Omega}$, then 
$$
T_\rho:=\inf_{x\in\Omega_\rho}T(x)
$$
is well defined.
\begin{Lem}[Local semiconcavity]\label{semi_value}
For any $\rho>0$,  any $x\in\Omega_\rho$ and any $z\in\R^n$ with $|z|<\frac \rho 8$, we have

\begin{equation}\label{eq:local_semiconcaveity}
	u(x+z)+u(x-z)-2u(x)\leqslant\frac{\bar{C}}{\rho}|z|^2,
\end{equation}
for some constant $\bar{C}>0$ independent of $x$ and $z$.
\end{Lem}

\begin{proof}
	Fix $x\in\Omega_\rho$, by Corollary \ref{exit_time_prob}, there is $y^*\in\partial\Omega$ such that 
	$$
	u(x)=g(y^*)+A_{T(x)}^\Omega(y^*,x).
	$$
	From \eqref{T-up}, we have
	$$
	0<T(x)\leqslant B.
	$$
	Let $\xi^*$ be a minimizer of $A_{T(x)}^\Omega(y^*,x)$.
	For any $z\in\R^n$ such that $x\pm z\in\Omega$, any $t^\pm>0$ and any $y^\pm\in\partial \Omega$, we have
	   \begin{align*}
		&u(x+z)+u(x-z)-2u(x)\\
		\leqslant&\, g(y^+)+g(y^-)-2g(y^*)+  A^{\Omega}_{t^+}(y^+,x+z)+A^{\Omega}_{t^-}(y^-,x-z)-2A^{\Omega}_{T(x)}(y^*,x).
	\end{align*}
	Taking $t^-=t^+=T(x)$ and $y^+=y^-=y^{*}$ in the inequality above, we obtain
	\begin{equation}\label{eq:semiconcave_1}
	\begin{split}
		&u(x+z)+u(x-z)-2u(x)\\
		\leqslant& A^{\Omega}_{T}(y^*,x+z)+A^{\Omega}_{T}(y^*,x-z)-2\int^0_{-T}L(\xi^*,\dot{\xi}^*)\ ds.
	\end{split}
	\end{equation}
	
	By Lemma \ref{eq:Lip_main}, we have
	$$
	 \operatorname*{ess\ sup}_{s\in[-t,0]}|\dot{\xi^*}(s)|\leqslant\kappa(CD/T(x)).
	$$
Thus, we have
    $$
	 \operatorname*{ess\ sup}_{s\in[-t,0]}|\dot{\xi^*}(s)|\leqslant\kappa(CD/T_\rho)=:K.
	$$
	Let $\tau=\min\{\frac \rho {4K},T_\rho\}$.
	We define two curves connecting $y^*$ and $x\pm z$ by
	$$
	\xi^{\pm}(s)=
	\begin{cases}
    \xi^*(s),&  s\in[-T(x),-\tau],\\
    \xi^*(s)\pm\frac {\tau+s}{\tau}z,& s\in(-\tau,0].
	\end{cases}
	$$
	We assert that $\xi^\pm(s)\subset \Omega$ for all $s\in[-T(x),0]$.  In fact, for each $s\in [-\tau,0]$, we get
	\begin{align*}
	|\xi^\pm(s)-x|&=|\xi^\pm(0)-\int_s^0\dot{\xi}^\pm(\sigma)d\sigma-x|\\
	&\leqslant |z|+	|\int_s^0(\dot{\xi}^*(\sigma)\pm\frac z \tau)d\sigma|\\
	&\leqslant 2|z|+\tau K\\
	&\leqslant \frac \rho 2.
	\end{align*}
 Thus, the above assertion is true. Hence, we have
 \begin{align*}
		&u(x+z)+u(x-z)-2u(x)\\
		\leqslant& \int^0_{-T(x)}L(\xi^+(s),\dot{\xi}^+(s))+L(\xi^-(s),\dot{\xi}^-(s))-2L(\xi^*(s),\dot{\xi}^*(s))\ ds\\
		=&\int^0_{-\tau}\bigg\{L\Big(\xi^*(s)+\frac {s+\tau}{\tau}z,\dot{\xi}^*(s)+\frac {z}{\tau}\Big)\\
		&\ \ \ \ \ \ \ \ +L\Big(\xi^*(s)-\frac {s+\tau}{\tau}z,\dot{\xi}^*(s)-\frac {z}{\tau}\Big)-2L(\xi^*(s),\dot{\xi}^*(s))\bigg\}\ ds\\
		\leqslant&\int^0_{-\tau}C_3\bigg(|\frac{s+\tau}{\tau}z|^2+|\frac{z}{\tau}|^2\bigg)\ ds\\
		\leqslant& C_4(\tau+\frac 1 \tau)|z|^2,
	\end{align*}
	where the positive constants $C_3$ and $C_4$ depend only on $K$ and $\rho$. Note that 
	$$
	\frac 1\tau=\max\left\{\frac{4K} \rho,\frac 1{T_\rho}\right\}.
	$$
	By Lemma \ref{upper_bound}, we have 
	$$
	\rho\leqslant d_{\partial \Omega}(x)\leqslant C_2T(x), \quad \forall x\in \Omega_\rho,
	$$
	which implies that 
	$$
	\frac 1{\tau}\leqslant \frac{4K} \rho+\frac 1{T_\rho}\leqslant \frac{4K} \rho+\frac {C_2}{\rho}=:\frac{C_5}{\rho}.
	$$
	Therefore,  we have
	$$
	u(x+z)+u(x-z)-2u(x)\leqslant C_4\left(T(x)+\frac {C_5}{\rho}\right)|z|^2\leqslant C_4\frac{BD+C_5}{\rho}|z|^2.
	$$
	\end{proof}

\begin{Cor}
	$u$ is a viscosity solution of Dirichlet problem \eqref{eq:HJ_BP}.
\end{Cor}
\begin{proof}
	By Proposition \ref{compatible_u} and Corollary \ref{exit_time_prob}, it is straightforward to check that $u$ satisfies equation $H(x,Du)=0$ a.e. on $\Omega$, which together with the local semiconcavity of $u$ obtained in Lemma \ref{semi_value}, implies that $u$ is a viscosity solution of Dirichlet problem \eqref{eq:HJ_BP}.
\end{proof}

\begin{Pro}\label{key1}
	For any $x\in\Omega$, $p\in D^*u(x)$ if and only if there exists a $u$-calibrated curve  $\xi=\xi_{x,p}:[-T_{x,p},0]\to\overline{\Omega}$ such that $\xi_{x,p}(-T_{x,p})\in\partial\Omega$, $\xi_{x,p}(0)=x$ and $p=\frac{\partial L}{\partial v}(\xi_{x,p}(0),\dot{\xi}_{x,p}(0))$.
\end{Pro}
In order to prove this proposition, we provide a preliminary result first.

\begin{Lem}\label{pre}
	If $u$ is differentiable at $x\in\Omega$, then there is a unique  $u$-calibrated curve $\gamma_x:[-T(x),0]\to \overline{\Omega}$ with $\gamma_x(0)=x$ and $Du(x)=\frac{\partial L}{\partial v}(x,\dot{\gamma}_x(0))$.
\end{Lem}

\begin{proof}
	In view of Corollary \ref{exit_time_prob},  there is $y_x\in \partial \Omega$ such that 
	$$
	u(x)=g(y_x)+A_{T(x)}^\Omega(y_x,x),
	$$
	which implies that there is a $C^2$ curve $\gamma_x\in \Gamma^{-T(x),0}_{y_x,x}(\overline{\Omega})$ such that 
	$$
	u(x)=g(y_x)+\int_{-T(x)}^0L(\gamma_x,\dot{\gamma}_x)ds.
	$$
	By Proposition \ref{compatible_u}, $\gamma_x$ is a $u$-calibrated curve and minimizes the quantity 
	$$
	g(\gamma(-T(x)))+\int_{-T(x)}^0L(\gamma,\dot{\gamma})ds,
	$$
	among all absolutely continuous curves $\gamma:[-T(x),0]\to\overline{\Omega}$ with $\gamma(0)=x$. By classical results in calculus of variations, we get $Du(x)=\frac{\partial L}{\partial v}(x,\dot{\gamma}_x(0))$.
\end{proof}

\begin{proof}[Proof of Proposition \ref{key1}]
	For any $x\in\Omega$, let $p\in D^*u(x)$. Then there exists a sequence $\{x_k\}_{k\in \N}\subset \Omega\setminus\{x\}$ such that $u$ is differentiable at $x_k$ and 
$$
\lim_{k\to+\infty}x_k=x,\quad \lim_{k\to+\infty}Du(x_k)=p.
$$ 
By Lemma \ref{pre}, for each $x_k$ there are $y_k\in\partial\Omega$ and $\gamma_k:[-T(x_k),0]\to\overline{\Omega}$ with $\gamma_k(0)=x_k$, $Du(x_k)=\frac{\partial L}{\partial v}(x_k,\dot{\gamma}_k(0))$ and 
$$
u(x_k)=g(y_k)+\int_{-T(x_k)}^0L(\gamma_k,\dot{\gamma}_k)ds.
$$
Let $(\gamma,\eta)$ be the solution of \eqref{eq:Ham_ODE} with initial conditions $\gamma(0)=x$, $\eta(0)=p$. Since $(\gamma_k,\frac{\partial L}{\partial v}(\gamma_k,\dot{\gamma}_k))$ converges to $(\gamma,\eta)$ locally uniformly, then $\gamma:[-T_{x,p},0]\to\overline{\Omega}$ is a $u$-calibrated curve for some $T_{x,p}>0$ with $\gamma(-T_{x,p})\in\partial \Omega$ and
$$
u(x)=g(\gamma(-T_{x,p}))+\int_{-T_{x,p}}^0L(\gamma,\dot{\gamma})ds.
$$

Since the other part of the proof is quite similar to the one of Theorem 8.4.14 in \cite{Cannarsa-Sinestrari}, we omit it here.
\end{proof}

\section{Global generalized characteristics}

In this section, we study the propagation of singularities of $u$, the value function of \eqref{eq:represent_formulae}, which is a viscosity solution of the Hamilton-Jacobi equation \eqref{eq:HJ_BP}. 

\begin{defn}
 We call $x\in\Omega$ a {\em singular point} if $u$ is not differentiable at $u$.	 The set of all singular points of $u$, denoted by $\SING$, is called the {\em singular set} of $u$. A point in $x\in\Omega$ is called a {\em cut point} of $u$, if for any $u$-calibrated curve $\xi:[a,b]\to M$ with $x\in\xi([a,b])$, we have $x=\xi(b)$. The set of all cut points of $u$, denoted by $\CUT$, is called the cut locus of $u$.
\end{defn}

It is a fact that $\SING\subset\CUT\subset\overline{\SING}$.

\subsection{Propagation of singularities of $u$}

In this part we assume $(L,\Omega,g)$ satisfies assumptions ({$\bf{SH}$}). To study the propagation of singularities of $u$, we will use the intrinsic methods developed in \cite{Cannarsa-Cheng3}. The basic idea is if a viscosity solution $u$ of \eqref{eq:HJ_BP} can have a representation as inf-convolution (like \eqref{eq:represent_formulae}), then the maximizers in the corresponding sup-convolution determine the propagation of singularities. More clearly, for fixed $x\in\Omega$ and suitable $\lambda>0$, we want to look for maximizers of
$$
u(y)-A^{\Omega}_t(x,y),\quad y\in B(x,\lambda t).
$$
Comparing to the problem in \cite{Cannarsa-Cheng3}, in this paper, we have to deal with some difficulties with the state constraint.

Now, fix any $x\in\Omega$, let $\delta_x=\frac 12 d_{\partial\Omega}(x)$. Then $\overline{B(x,\delta_x)}\subset\Omega$. It is known that there exists $\lambda_0>0$ depending only on $L$ and $\text{Lip}\,(u)$ such that each maximizer $y_{t,x}$ of the function $u(\cdot)-A^{\Omega}_t(x,\cdot)$ is contained in the ball $\overline{B(x,\lambda_0t)}$ for any $t>0$ (see \cite[Lemma 3.1]{Cannarsa-Cheng3}). Fix $\lambda>\lambda_0$, choose $t_x>0$ such that 
\begin{equation}\label{eq:relation1}
	t_x<\frac 1{2\lambda}d_{\partial\Omega}(x).
\end{equation}
Thus, $\overline{B(x,\lambda t)}\subset\overline{B(x,\delta_x)}$ for $0<t\leqslant t_x$.

\begin{Lem}\label{prop_sing_no_uniquness}
For any $x\in\CUT$ and any $\lambda>\lambda_0$, suppose $t_x>0$ is chosen such that \eqref{eq:relation1} holds. If $t_x$ also satisfies the following relation
	\begin{equation}\label{eq:relation2}
		t_x<T(x),
	\end{equation}
	then, for any $t\in(0,t_x]$, each maximizer of the function
	\begin{equation}\label{eq:maximizer}
		y\mapsto u(y)-A^{\Omega}_t(x,y)
	\end{equation}
	is contained in $\SING$. 
\end{Lem}

\begin{proof}
	For any $x\in\Omega$, any $\lambda>\lambda_0$ and any $0<t<t_x$, let $y_{t,x}$ be any maximizer of \eqref{eq:maximizer}. It clear that $y_{t,x}\in B(x,\lambda t)$ by \eqref{eq:relation1}. 
	
	If $x\in\CUT$, then we will show $y_{t,x}\in\SING$ for all $t\in(0,t_x]$. Assume  by contradiction that $y_{t,x}$ is a point of differentiability of $u$ for some $t\in(0,t_x]$. Thus 
	$$
	0\in D^+\{u(\cdot)-A^{\Omega}_t(x,\cdot)\}(y_{t,x})=Du(y_{t,x})-D^-\{A^{\Omega}_t(x,\cdot)\}(y_{t,x}),
	$$
	equivalently, $Du(y_{t,x})\in D^-\{A^{\Omega}_t(x,\cdot)\}(y_{t,x})$. It follows that $A^{\Omega}_t(x,\cdot)$ is differentiable at $y_{t,x}$ and
	\begin{align}\label{end-point}
	p_{t,x}:=Du(y_{t,x})=D_yA^{\Omega}_t(x,y_{t,x})
	\end{align}
	since $A^{\Omega}_t(x,\cdot)$ is locally semiconcave (see Proposition \ref{C11_A_t}). Hence, there is a unique minimizer of $A^{\Omega}_t(x,y_{t,x})$, denoted by $\xi_{t,x}:[-t,0]\to\Omega$, such that 
	$$
	D_yA^{\Omega}_t(x,y_{t,x})=\frac{\partial L}{\partial v}(\xi_{t,x}(0),\dot{\xi}_{t,x}(0)).
	$$ 
By Proposition \ref{key1}, there exists a $C^2$ $u$-calibrated curve $\gamma_x:[-T_{x,p_{t,x}},0]\to\Omega$ such that $\gamma_x(0)=y_{t,x}$ and $p_{t,x}=\frac{\partial L}{\partial v}(\gamma_x(0),\dot{\gamma}_x(0))$. Since $\xi_{t,x}(0)=\gamma_x(0)$ and \eqref{end-point}, then $\xi_{t,x}$ and $\gamma_x$ coincide on $[-t,0]$ since $t<t_x<T(x)\leqslant T_{x,p_{t,x}}$ by \eqref{eq:relation2}. This leads to a contradiction since $x=\gamma_x(-t)$ and $\gamma_x(-t)\not\in\CUT$. 
\end{proof}

To ensure the uniqueness of the maximizer of $u(\cdot)-A^{\Omega}_t(x,\cdot)$ in $B(x,\lambda t)$, we need more work by using the semiconcavity and convexity estimate of $u(\cdot)$ and $A^{\Omega}_t(x,\cdot)$ in the ball $B(x,\lambda t)$ (see Proposition \ref{C11_A_t}). 


\begin{The}\label{propagation1}
Let $x_0\in\CUT$. Then, we have
\begin{enumerate}[\rm (a)]
	\item either, there exists a generalized characteristic $\mathbf{x}:[0,+\infty)\to\Omega$ starting from $\mathbf{x}(0)=x_0$ such that $\mathbf{x}(s)\in\SING$ for all $s\in[0,+\infty)$,
	\item or, there exist $T>0$ and a generalized characteristic $\mathbf{x}:[0,T)\to\Omega$ starting from $\mathbf{x}(0)=x_0$ such that $\mathbf{x}(s)\in\SING$ for all $s\in[0,T)$, and a sequence of positive real numbers $\{s_k\}$ such that
	$$
	\lim_{k\to\infty}s_k=T,\quad\text{and}\quad\lim_{k\to\infty}d_{\partial\Omega}(\mathbf{x}(s_k))=0.
	$$
\end{enumerate}
\end{The}

\begin{proof}
	We begin with a local study. Let $x_0\in\Omega$ be a cut point of $u$. By Proposition \ref{compatible_u}, $u$ is Lipschitz on $\overline{\Omega}$. From \cite[Lemma 3.1]{Cannarsa-Cheng3}, there exists $\lambda_0>0$ depending only on $L$ and $\text{Lip}\,(u)$ such that each maximizer of the function $u(\cdot)-A^{\Omega}_t(x_0,\cdot)$ is contained in the ball $\overline{B(x_0,\lambda_0t)}\subset\Omega$. Fix any $\lambda>\max\{\lambda_0,\frac{C_2}{2}\}$, where $C_2$ is the constant in Lemma \ref{upper_bound}, we define
	\begin{equation}
	\begin{split}
		t^*_0=&\min\frac 12\left\{\frac{1}{2\lambda}d_{\partial\Omega}(x_0),\frac {d_{\partial\Omega}(x_0)}{\kappa(\lambda)}, t_{\lambda}, \frac {C(\lambda)d_{\partial\Omega}(x_0)}{\bar{C}}\right\}\\
		\leqslant&\min\frac 12\left\{T(x_0),\frac {d_{\partial\Omega}(x_0)}{\kappa(\lambda)}, t_{\lambda}, \frac {C(\lambda)d_{\partial\Omega}(x_0)}{\bar{C}}\right\},
	\end{split}
	\end{equation}
	where $C(\lambda)$ is the constant in Proposition \ref{C11_A_t}, $\bar{C}$ is the constant in Lemma \ref{semi_value}, and $\kappa$ is the function in Lemma \ref{eq:Lip_main}.
	Notice that the inequality above follows from Lemma \ref{upper_bound}. Then, for any $t\in(0,t^*_0]$, the function defined by
	$$
	y\mapsto u(y)-A^{\Omega}_t(x_0,y),\quad y\in B(x_0,\lambda t)\subset\Omega,
	$$
	is strictly concave, by Proposition \ref{C11_A_t} and Lemma \ref{semi_value}, with a unique maximizer $y^0_{t,x_0}\in B(x_0,\lambda t)$. We define the curve
	\begin{equation}
		\mathbf{x}_0(s)=\begin{cases}
			y^0_{s,x_0},& s\in(0,t^*_0],\\
			x_0,& s=0.
		\end{cases}
	\end{equation}
	By Lemma \ref{prop_sing_no_uniquness} and the same discussion in \cite{Cannarsa-Cheng3}, $\mathbf{x}_0:[0,t^*_0]\to\Omega$ is a generalized characteristic and $\mathbf{x}_0(s)\in\SING$ for all $s\in [0,t^*_0]$. We denote $x_1=\mathbf{x}_0(t^*_0)$ and we try to extend the generalized characteristic $\mathbf{x}_0$ starting from $x_1$ since $x_1\in\SING\subset\CUT$. 
	
	By the definition of $t^*_0$, we observe that if 
	$$
	d_{\partial\Omega}(x_0)\leqslant t_{\lambda},
	$$
	then there exists $K_1(\lambda)>0$ such that $t^*_0\geqslant K_1(\lambda)d_{\partial\Omega}(x_0):=\bar{t}_0$. Thus, replacing $t^*_0$ by $\bar{t}_0$, we obtain a singular generalized characteristic, also denoted by $\mathbf{x}_0$, defined on $[0,\bar{t}_0]$.
	
	Now, by deduction, for each $k\in\N$, we define 
	$$
	t^*_k=\min\frac 12\left\{\frac{1}{2\lambda}d_{\partial\Omega}(x_k),\frac {d_{\partial\Omega}(x_k)}{\kappa(\lambda)}, t_{\lambda}, \frac {C(\lambda)d_{\partial\Omega}(x_k)}{\bar{C}}\right\}.
	$$
	A curve $\mathbf{x}_k:[0,t^*_k]\to\Omega$ is defined as follows: let $\mathbf{x}_k(0)=x_{k}$ and for any $s\in(0,t^*_k]$, let $\mathbf{x}_k(s)$ be the unique maximizer of the function $y\mapsto u(y)-A^{\Omega}_s(x_{k},y)$ with $x_{k}=\mathbf{x}_{k-1}(t^*_{k-1})$ since this function is strictly concave. It is clear that $\mathbf{x}_k(s)\in\SING$ for all $s\in[0,t^*_k]$ since $x_{k}\in\CUT$. 
	
	Suppose $\limsup_{k\to\infty}t^*_k=\bar{t}>0$, we define $\sum_{i=0}^{k-1}t^*_i=s_k$ and $s_0=0$, then $\sum^{\infty}_{i=0}t^*_i=+\infty$. Thus, $\mathbf{x}:[0,+\infty)\to\Omega$ defined by,
	\begin{equation}\label{eq:global_gc}
		\mathbf{x}(t)=\mathbf{x}_k(t-s_k), \quad t\in[s_k,s_{k+1}],\quad k=0,1,\ldots,
	\end{equation}
	is an expected generalized characteristic starting form $x_0$ and $\mathbf{x}(t)\in\SING$ for all $t>0$. 
	
	Now, suppose $\limsup_{k\to\infty}t^*_k=0$. Then, without loss of generality, we can suppose that there exists $K_2(\lambda)>0$ such that
	$$
	t^*_k=K_2(\lambda)d_{\partial\Omega}(x_k)\leqslant t_{\lambda}, 	$$
	for $k$ large enough. Therefore, $\mathbf{x}:[0,\sum^{\infty}_{i=0}t^*_i)\to\Omega$ defined by \eqref{eq:global_gc} is a singular generalized characteristic and $d_{\partial\Omega}(\mathbf{x}(s_k))=d_{\partial\Omega}(x_{k+1})\to0$ as $k\to\infty$. Let $T=\sum^{\infty}_{i=0}t^*_i$, then $T\leqslant+\infty$. Finally, notice that if $T=+\infty$ we also obtain a singular global generalized characteristic defined on $[0,+\infty)$. This completes the proof.
\end{proof}

\subsection{Further results for mechanical systems} 

For general Tonelli systems, Theorem \ref{propagation1} shows that there is a generalized characteristic $\mathbf{x}$ starting from a cut point of $u$ which stays on $\SING$ such that $\mathbf{x}$ can be extended to $(0,+\infty]$ or it will hit the boundary $\partial\Omega$. In this section, we will show that for a certain family of mechanical systems, we can exclude the possibility that the singularities approach the boundary. 

In this part, we consider the following Lagrangians on $\R^n$:
\begin{equation}\label{eq:mech}
	L(x,v)=\frac 12\langle A(x)v,v\rangle-\langle DS(x),v\rangle-V(x),\quad (x,v)\in\R^n\times\R^n,
\end{equation}
where $A(x)$ is a symmetric and positive definite matrix $C^2$ depending on $x$, $S$ (resp. $V$) is a $C^3$ (resp. $C^2$) function on $\R^n$. Let $\Omega\subset \R^n$ be a bounded Lipschitz domain. Assume that
\begin{equation}\label{eq:condition_M}
	\ \max_{x\in\overline{\Omega}}V(x)<0\quad \text{and}\quad g+S\ \text{is constant on}\ \partial\Omega.
\end{equation}
Let
\begin{equation}\label{eq:mech2}
	L_0(x,v)=L(x,v)+\langle DS(x),v\rangle=\frac 12\langle A(x)v,v\rangle-V(x).
\end{equation}
Denote by $H_0$ the Hamiltonian associated with $L_0$. It is a fact that $c_{\Omega}(L)=c_{\Omega}(L_0)=\max_{x\in\overline{\Omega}}V(x)$.   We suggest readers see Proposition \ref{exact_form} first before Theorem \ref{thm_mech}.

Throughout this section, let $u$ be the value function of \eqref{eq:represent_formulae} with respect to $(L,g)$ which is a viscosity solution of \eqref{eq:HJ_BP}. 

\subsubsection{Global propagation of singularities of $u$}


\begin{The}\label{thm_mech}
Let $v=u+S$. If $x_0\in\CUT$, then there exists a unique generalized characteristic $\mathbf{x}:[0,+\infty)\to\Omega$ with $\mathbf{x}(0)=x_0$ of $H_0(x,Dv)=0$, i.e., $\mathbf{x}$ is a Lipschitz curve with $\mathbf{x}(0)=x_0$ and satisfies
\begin{equation}\label{eq:generalized_gradient}
	\dot{\mathbf{x}}^+(s)\in A^{-1}(\mathbf{x}(s))D^+v(\mathbf{x}(s)),\quad \forall s\in[0,+\infty).
\end{equation}
Moreover,  $\mathbf{x}(s)\in\SING$ for all $s\in[0,+\infty)$.
\end{The}

\begin{proof}
Note that the singular sets of $u$ and $v$ are the same.
From Proposition \ref{exact_form}, we have
$$
v(x)=\inf_{y\in\partial \Omega}\{g(y)+S(y)+\Phi_{L_0}^\Omega(y,x)\}.
$$
Without loss of generality, we assume $g+S\equiv0$.
For any $x_0\in\Omega$, we define
	$$
	\Lambda_{x_0}=\{x\in\overline{\Omega}: v(x)\geqslant v(x_0)\}.
	$$
	Then $\Lambda_{x_0}$ is compact and we assert that $\partial\Lambda_{x_0}=\{x\in\overline{\Omega}: v(x)=v(x_0)\}\subset\Omega$. Indeed, if there exists $x\in\partial\Lambda_{x_0}\cap\partial\Omega$, then there exists $y_0\in\partial\Omega$ such that
	\begin{align*}
		v(x_0)=g(y_0)+S(y_0)+\Phi^{\Omega}_{L_0}(y_0,x_0)=v(x)=g(x)+S(x)=0,
	\end{align*}
	since $g+S\equiv0$ and $v\vert_{\partial\Omega}=g+S$. Invoking Corollary \ref{exit_time_prob} and the symmetric property of $L$ in $v$-variable, we have that, there exists $T>0$ such that
	\begin{align*}
		\Phi^{\Omega}_{L_0}(y_0,x_0)=A^{\Omega,L_0}_T(y_0,x_0)=A^{\Omega,L_0}_T(x_0,y_0)=0.
	\end{align*}
	In view of the energy condition $c_{\Omega}(L)=c_{\Omega}(L_0)<0$ together with the equalities above, we obtain that
	\begin{align*}
		0<\frac 1{2T}A^{\Omega,L_0}_{2T}(x_0,x_0)\leqslant 0,
	\end{align*}
	which leads to a contradiction. 
	
	
	For any cut point $x_0\in\Omega$, let $r_0=\frac 12d(\partial\Omega,\Lambda_{x_0})$. Then, due to Proposition \ref{compatible_u}  and Lemma \ref{semi_value}, $v$ is Lipschitz on $\overline{\Omega}$ and semiconcave on each $B(x,r_0)\subset\Omega$ with a constant of semiconcavity $C_{r_0}$ for all $x\in\Lambda_{x_0}$ uniformly. Fix $\lambda>\max\{\lambda_0,\frac{C_2}{2}\}$, where  $C_2$ is the constant in Lemma \ref{upper_bound} and $\lambda_0>0$ is the the constant for which each maximizer $y_{t,x}$ of the function $v(\cdot)-A^{\Omega,L_0}_t(x_0,\cdot)$ is contained in the ball $\overline{B(x_0,\lambda_0t)}$. We define 
	$$
	t^*_{0}=\min\frac 12\left\{\frac{r_0}{2\lambda},\frac {r_0}{\kappa(\lambda)}, t_{\lambda}, \frac {C(\lambda)r_0}{\bar{C}}\right\}\leqslant\min\frac 12\left\{T(x_0),\frac {r_0}{\kappa(\lambda)}, t_{\lambda}, \frac {C(\lambda)r_0}{\bar{C}}\right\},
	$$
	where $C(\lambda)/t$ is the uniform convexity constant of $A^{\Omega}_t(x_0,\cdot)$ on $B(x_0,\lambda t)$ by Proposition \ref{C11_A_t}, $\bar{C}$ is the constant in Lemma \ref{semi_value} and $\kappa$ is the function in Lemma \ref{eq:Lip_main}. The above inequality follows from Lemma \ref{upper_bound}. Therefore, $v(\cdot)-A^{\Omega,L_0}_t(x_0,\cdot)$ is strictly concave on $B(x_0,\lambda t)$ for all $0<t\leqslant t^*_0$ by the definition of $t^*_0$, and then there exists a unique maximizer $y^0_{t,x}$ of $v(\cdot)-A^{\Omega,L_0}_t(x_0,\cdot)$ which is contained in $\SING$ by Lemma \ref{prop_sing_no_uniquness}. Define
	$$
	\mathbf{x}_{0}(t)=\begin{cases}
		y^0_{t,x},& t\in(0,t^*_0],\\
		x_0,& t=0.
	\end{cases}
	$$
	Using the same arguments in \cite{Cannarsa-Cheng3}, $\mathbf{x}_{0}:[0,t^*_0]\to\Omega$ is Lipschitz and $\mathbf{x}_{0}$ is a generalized characteristic starting from $x_0$. 
	
	To obtain a global extension, we should use the monotonicity property of $v$ along the generalized characteristics (see, for instance, \cite[Theorem 3.7]{Cannarsa-Cheng}), i.e.,
	$$
	v(\mathbf{x}_{0}(t_1))\leqslant v(\mathbf{x}_{0}(t_2)),\quad t_1<t_2.
	$$
	Let $x_1=\mathbf{x}_0(t^*_0)$. Then $\Lambda_{x_1}\subset\Lambda_{x_0}$ and $r_1=\frac 12d(\partial\Omega,\Lambda_{x_1})\geqslant r_0$. Thus, define
	$$
	t^*_{1}=t^*_0=\min\frac 12\left\{\frac{r_0}{2\lambda},\frac {r_0}{\kappa(\lambda)}, t_{\lambda}, \frac {C(\lambda)r_0}{\bar{C}}\right\}\leqslant\min\frac 12\left\{T(x_1),\frac {r_1}{\kappa(\lambda)}, t_{\lambda}, \frac {C(\lambda)r_1}{\bar{C}}\right\},
	$$
	and $v(\cdot)-A^{\Omega,L_0}_t(x_1,\cdot)$ has a unique maximizer $y^1_{t,x}$ in $B(x_1,\lambda t)$ for all $0<t\leqslant t^*_1$. Define
	$$
	\mathbf{x}_{1}(t)=\begin{cases}
		y^1_{t,x},& t\in(0,t^*_1],\\
		x_1,& t=0,
	\end{cases}
	$$
	and $\mathbf{x}_{1}:[0,t^*_1]\to\Omega$ is a generalized characteristic starting from $x_1$ and $\mathbf{x}_1(t)\in\SING$ for all $t\in[0,t^*_1]$. Deductively, for any nonnegative integer $k$, there exists $\mathbf{x}_{k}:[0,t^*_k]\to\Omega$ being Lipschitz and $\mathbf{x}_{k}$ is a generalized characteristic starting from $x_k=\mathbf{x}_{k-1}(t^*_{k-1})$. Define $\sum_{i=0}^{k-1}t^*_i=s_k$ and $s_0=0$, then $\sum^{\infty}_{i=0}t^*_i=+\infty$ since $t^*_k=t^*_{k+1}$ for all $k$. Thus, $\mathbf{x}:[0,+\infty)\to\Omega$ defined by,
	\begin{align*}
		\mathbf{x}(t)=\mathbf{x}_k(t-s_k), \quad t\in[s_k,s_{k+1}],  \ k=0,1,\ldots,
	\end{align*}
	is an expected singular generalized characteristic starting form $x_0$, and   \eqref{eq:generalized_gradient} follows from \cite[Proposition 3.6]{Cannarsa-Cheng}.
\end{proof}

\begin{Rem}\label{mech_general}
For the aforementioned mechanical systems, the generalized characteristic \eqref{eq:generalized_gradient} can produce a semiflow and $\Lambda_x$ is an invariant set of the semiflow. Therefore, the associated generalized characteristic cannot hit the boundary forever (compare to the statement in Theorem \ref{propagation1}).
\end{Rem}

\subsubsection{Topology of cut locus of $u$}

In the context of classical weak KAM theory, the topology of $\CUT$ and $\SING$ with respect to a weak KAM solution $u$ on compact manifold has been studied in \cite{Cannarsa-Cheng-Fathi}. In this section, we will explain certain techniques used in \cite{Cannarsa-Cheng-Fathi} can also be applied to the value function $u$ of \eqref{eq:represent_formulae}.

\begin{Lem}\label{homotopy}
 let $L$ be as in \eqref{eq:mech}. There exists  a (continuous) homotopy $F:M\times [0,1]\to\Omega$, with the following properties:
\begin{enumerate}[\rm (a)]
  \item for all $x\in\Omega$, we have $F(x,0)=x$;
  \item if $F(x,s)\not\in\SING$ for some $s>0$ and $x\in\Omega$, then the curve $\sigma\mapsto F(x,\sigma)$ is $u$-calibrated on $[0,s]$;
  \item if there exists a $u$-calibrated curve $\xi:[0,s]\to\Omega$ with $\xi(0)=x$, then $\sigma\mapsto F(x,\sigma)=\xi(\sigma)$ for every $\sigma\in [0,\min\{s,1\}]$.
\end{enumerate}
\end{Lem}

\begin{proof}
	For any $x\in\Omega$, let $\mathbf{x}_x$ be the unique generalized characteristic starting from $x$ defined by \eqref{eq:generalized_gradient} (for the uniqueness of the solution of the differential inclusion \eqref{eq:generalized_gradient}, the readers can refer to \cite{Cannarsa-Cheng}). Then, the homotopy is defined by $F(x,\sigma)=\mathbf{x}_x(\sigma)$ for all $\sigma\in[0,1]$.
	
	(a) follows from the definition of $F$ directly and (c) follows from the fact that a $u$-calibrated curve is an extremal curve and it also satisfies \eqref{eq:generalized_gradient}.
	
	Now, we turn to the proof of (b). For each $x\in\Omega$, the generalized characteristic $\mathbf{x}_x$ has the form
	\begin{align*}
		\mathbf{x}_x(t)=\mathbf{x}_k(t-s_k), \quad t\in[s_k,s_{k+1}],\ k=0,1,\ldots,
	\end{align*}
	as in the proof of Theorem \ref{thm_mech}, since the uniqueness of the solution of \eqref{eq:generalized_gradient} with $\mathbf{x}_x(0)=x$. Recall that $t^*_k=t^*_{k+1}$ (for $k=0,1\ldots$), $\sum_{i=0}^{k-1}t^*_i=s_k$ and $s_0=0$. Furthermore, $\mathbf{x}_k(\sigma)$ is the unique maximizer of the function $u(\cdot)-A_{\sigma}(\mathbf{x}_k(0),\cdot)$ for $\sigma\in(0,t^*_k]$. Now, suppose $F(x,s)\not\in\SING$ for some $s>0$. Then there exist $k_0\in\N$ and $\sigma_0\in(0,t^*_{k_0}]$ such that $s=s_{k_0}+\sigma_0$ and $\mathbf{x}_{k_0}(\sigma_0)\not\in\SING$, it follows $\mathbf{x}_{k_0}(0)\not\in\SING$ by Lemma \ref{prop_sing_no_uniquness} and $\mathbf{x}_k:[0,\sigma_0]\to\Omega$ is a $u$-calibrated curve. By deduction, it shows that $\mathbf{x}_k:[0,t^*_k]\to\Omega$ is a $u$-calibrated curve. Finally, by the uniqueness of the solution of \eqref{eq:generalized_gradient} again, the curve $\sigma\mapsto F(x,\sigma)$ is $u$-calibrating on $[0,s]$. This completes the proof.
\end{proof}

\begin{The}\label{topology_thm_1}
The inclusion $\SING\subset\CUT\subset\BSING\cap\Omega\subset\Omega$ are all homotopy equivalences. Moreover, for every connected component $C$ of $\Omega$ the three intersections $\SING\cap C$, $\CUT\cap C$, and $\BSING\cap C$ are path-connected.
\end{The}

\begin{The}\label{topology_thm_2}
The spaces $\SING$ and $\CUT$ are locally contractible, i.e., for every $x\in\SING$ (resp. $x\in\CUT$) and every neighborhood $V$ of $x$ in $\SING$ (resp. $\CUT$), we can find a neighborhood $W$ of $x$ in $\SING$ (resp. $\CUT$), such that $W\subset V$ and $W$ in null-homotopic in $V$. 

Therefore, $\SING$ and $\CUT$ are locally path connected.	
\end{The}

The proofs of Theorem \ref{topology_thm_1} and Theorem \ref{topology_thm_2} are based on the homotopy constructed by Lemma \ref{homotopy} and they are quite similar to the ones for Theorem 1.1, Corollary 1.2 and Theorem 1.3 in \cite{Cannarsa-Cheng-Fathi}. The only difference is to replace $M\setminus\IU$ (in \cite{Cannarsa-Cheng-Fathi}) by $\Omega$ where $\IU$ is the \textit{projected Aubry set with respect to $u$}. Furthermore, the topological results in Theorem \ref{topology_thm_1} and \ref{topology_thm_2} also hold under the same assumptions of  Theorem \ref{propagation-global} by the similar reason.

\subsubsection{On the critical points of $u$} 
For $L$ be defined by \eqref{eq:mech}, Theorem \ref{thm_mech} shows that if $x\in\CUT$, then the unique generalized characteristic $\mathbf{x}:[0,+\infty)\to\Omega$ defined by \eqref{eq:generalized_gradient} governs the propagation of singularities of $u$. Now, let us recall some basic facts of such generalized characteristics (see, for instance, \cite{Cannarsa-Cheng} for the proof).

\begin{Pro}\label{properties_g_c}
Let $L=L_0$ be as in \eqref{eq:mech2} and let $\mathbf{x}:[0,+\infty)\to\Omega$ be defined by \eqref{eq:generalized_gradient} with $\mathbf{x}(0)=x_0$. Then we have the following properties:
\begin{enumerate}[\rm (a)]
  \item if $x$ belongs to $\SING$, then $\mathbf{x}$ is a unique solution of \eqref{eq:generalized_gradient} such that $\mathbf{x}(0)=x$ and $\mathbf{x}(s)\in\SING$ for all $s\in [0,+\infty)$;
  \item $\mathbf{x}$ is Lipschitz, the right derivative $\dot{\mathbf{x}}^+(s)$ exists for all $s\in[0,+\infty)$, and
\begin{equation*}\label{generalized_characteristics_mech_sys}
\dot{\mathbf{x}}^+(s)=A^{-1}(\mathbf{x}(s))p(s),\quad\forall s\in[0,+\infty),
\end{equation*}
where $p(s)$ is the unique point of $D^+u(\mathbf{x}(s))$ such that
\begin{equation}\label{minimality}
\langle A^{-1}(\mathbf{x}(s))p(s),p(s)\rangle=\min_{p\in D^+u(\mathbf{x}(s))}\langle A^{-1}(\mathbf{x}(s))p,p\rangle.
\end{equation}
Moreover, $\dot{\mathbf{x}}^+(s)$ is right-continuous;
   \item the right derivative of $u(\mathbf{x}(\cdot))$ exists on $[0,+\infty)$ and is given by
   \begin{equation}\label{eq:dv_c}
   	\frac d{ds^+}u(\mathbf{x}(s))=\langle p(s),A^{-1}(\mathbf{x}(s))p(s)\rangle,\quad s\in[0,+\infty).
   \end{equation}
\end{enumerate}
\end{Pro}

\begin{defn}[Critical point]
	We say that $x\in\Omega$ is a critical point of $u$, if $0\in  \mathrm{co}\ \frac{\partial H}{\partial p}(x,D^+u(x)).$
\end{defn}
From Proposition \ref{properties_g_c}, it is clear that the propagation will halt at a critical point.

\begin{The}\label{th:critical1}
Let $L$ be defined by \eqref{eq:mech}. If $\mathcal{K}$ is a connected component of $\SING$, then $\mathcal{K}$  contains a critical point of $u$. 
\end{The}

\begin{proof}
It is sufficient to prove our statement for $L=L_0$ which is given by \eqref{eq:mech2}. Note that $0\in\mathrm{co}\ \frac{\partial H_0}{\partial p}(x,D^+u(x))$ if and only if $0\in D^+u(x)$.
Fix  $x_0\in \mathcal{K}$ and let $\mathbf{x}:[0,+\infty)\to\Omega$ be the generalized singular characteristic with initial point $x_0$. Then $\mathbf{x}(t)\in \mathcal{K}$ for all $t\in [0,+\infty)$ and, by \eqref{eq:dv_c} we have
\begin{equation}
\label{eq:increase_of_v}
u(\mathbf{x}(t))-u(x_0)=\int_0^t\langle p(s),A^{-1}(\mathbf{x}(s))p(s)\rangle\,ds,\quad\forall t\geqslant 0,
\end{equation}
where $p(\cdot)$ satisfies \eqref{minimality}.

If $\mathbf{x}(t)$ is  critical of $u$ for some $t\geqslant 0$, then the conclusion has been proved. 

Suppose that $\mathbf{x}(t)$ is not a critical point for all $t\geqslant 0$. Then we assert
\begin{equation}
\label{eq:abs-abs}
\lim_{j\to\infty}\langle p(s_j),A^{-1}(\mathbf{x}(s_j))p(s_j)\rangle=0
\end{equation}
for some sequence $\{s_j\}$. 
Suppose not.  We have  
\begin{equation*}
\delta:=\inf_{s\geqslant 0} \langle p(s),A^{-1}(\mathbf{x}(s))p(s)\rangle>0\,,
\end{equation*}
then, appealing to \eqref{eq:increase_of_v}, we obtain
 \begin{equation*}
u(\mathbf{x}(t))-u(x_0)\geqslant \delta t, \quad\forall t\geqslant 0\,,
\end{equation*}
which contradicts the fact that $u$ is bounded on $\mathcal{K}$. Thus, \eqref{eq:abs-abs} is true.

Now, since $\mathcal{K}$ is compact and the set-valued map $x\leadsto D^+u(x)$ is upper semicontinuous, 
if necessary passing to a subsequence,we have $\mathbf{x}(s_j)\to \bar x\in \mathcal{K}$  and
 $p(s_j)\to \bar p\in D^+u(\bar x)$ as $j\to\infty$. Thus, 
 $\langle \bar p,A^{-1}(\bar x)\bar p\rangle=0$ by \eqref{eq:abs-abs}.
 Since  $A^{-1}(\bar x)$ is positive definite, then $\bar p=0$ and $\bar x$ is a critical point of $u$. 
\end{proof}

\section{More on global propagation of singularities} 
In the last section, under the following additional assumptions we will provide a global propagation result for general Tonelli Lagrangian systems.
In view of Theorem \ref{propagation1}, it suffices to show the semiconcavity and local semiconvexity of $u$ near $\partial \Omega$ for this purpose.

\begin{enumerate}[(\bf{G}1)]
    \item  There is $\nu\in[0,1)$ such that $g(y_1)-g(y_2)\leqslant \nu\Phi^{\Omega}_L(y_2,y_1)$, $\forall y_1,\, y_2\in\partial \Omega$.
  \item There exists $G\in C^{1,1}(\Gamma_\delta)$ for some $\delta>0$, such that $g=G|_\Gamma$, and  for any $x$, $y\in\Gamma$, we have
  \begin{align}\label{g2}
  \langle \nabla G(x),x-y\rangle\leqslant \breve{C}|x-y|^2
  \end{align}
  for some $\breve{C}>0$ independent of $x$ and $y$, where $\Gamma=\partial\Omega$ and $\Gamma_\delta$ denotes the $\delta$-neighborhood of $\Gamma$.
\end{enumerate}

\begin{Rem}
We will take a closer look at conditions ({\bf G1}), ({\bf G2}):
\begin{enumerate}[$\bullet$]
	\item if $g$ is constant on $\partial\Omega$, then conditions ({\bf G1}) and ({\bf G2}) hold;
	\item condition ({\bf G2}) implies that: there exist $K_1$, $K_2>0$ such that for any $y_1$, $y_2$, $\bar{y}\in\partial\Omega$, we have
	\begin{equation}\label{eq:g_semiconcave}
		g(y_1)+g(y_2)-2g(\bar{y})\leqslant K_1|y_1-y_2|^2+K_2|y_1+y_2-2\bar{y}|;
	\end{equation}
	\item since $\langle \nabla G(x),y-x\rangle=\langle \nabla G(y),y-x\rangle+\langle\nabla G(x)-\nabla G(y),y-x\rangle\leqslant C|x-y|^2$ for some $C>0$, then condition \eqref{g2} is equivalent to $\langle \nabla G(x),y-x\rangle\leqslant \breve{C}|x-y|^2$ for all $x$, $y\in \partial\Omega$. 
\end{enumerate}
\end{Rem}

The main result of this section is stated as follows.

\begin{The}\label{propagation-global}
Let $\Omega\subset\R^n$ be a bounded domain with $C^2$ boundary,
let $L$ be a Tonelli Lagrangian satisfying  $L\geqslant\alpha >0$ and let $g$ satisfy ({\bf G1}),({\bf G2}).
If $x_0\in\CUT$, then there exists a generalized characteristic $\mathbf{x}:[0,+\infty)\to\Omega$ starting from $\mathbf{x}(0)=x_0$ such that $\mathbf{x}(s)\in\SING$ for all $s\in[0,+\infty)$.
\end{The}
This theorem is a direct consequence of Theorem \ref{propagation1}, Proposition  \ref{global_semi_value} and Proposition \ref{semiconvex}.

\subsection{Semiconcavity of $u$ up to the boundary}
In order to get the semiconcavity of $u$ on $\overline{\Omega}$, we need certain additional conditions. 

\begin{Pro}[Global semiconcavity]\label{global_semi_value}
Let $L$ be a Tonelli Lagrangian with $L\geqslant \alpha>0$, 
let $g$ satisfy ({\bf G1}) and ({\bf G2}) and 
let $\Omega$ be a bounded Lipschitz domain satisfying
 the exterior sphere condition: there exists $\rho>0$ such that
  $$
  	\forall x\in\partial\Omega, \ \exists x_0\in\R^n\setminus\Omega\ \text{such that}\ x\in\overline{B(x_0,\rho)}\subset\R^n\setminus\Omega.
  $$
	Then $u$ is semiconcave on $\overline{\Omega}$.
\end{Pro}

\begin{Rem}
Since $L\geqslant\alpha>0$, then $c_\Omega(L)<0$ and $\Phi_L^\Omega(\cdot,\cdot)$ is a nonnegative Lipschitz function with a Lipschitz constant $C_1=\theta_2(1)C$, where $C>0$ is a constant depending only on $\Omega$.  
\end{Rem}

\begin{proof}[Proof of Proposition \ref{global_semi_value}]
	In order to give the semiconcavity estimate of $u$ on $\overline{\Omega}$, it suffices to show that there exists $C\geqslant 0$ such that 
$$
u(x+h)+u(x-h)-2u(x)\leqslant C|h|^2
$$
for all $x\in\overline{\Omega}$ and $h\in\R^n$ such that $[x-h,x+h]\subset\overline{\Omega}$. In the rest of the proof we use $c_i$ to denote certain positive constants independent of $x$ and $h$.
For any $x\in\overline{\Omega}$, there exist $\bar{y}\in\partial\Omega$ and a curve $\xi:[-T(x),0]\to \R^n$ such that
$$
u(x)=g(\bar{y})+\int^0_{-T(x)}L(\xi,\dot{\xi})\ ds=g(\bar{y})+\Phi^{\Omega}_L(\bar{y},x).
$$
It means that $\xi$ is a minimizer of $A^{\Omega}_{T(x)}(\bar{y},x)=\Phi^{\Omega}_L(\bar{y},x)$. Define $\xi_h=\xi+h$ and $\xi_{-h}=\xi-h$. Then $\xi_h(0)=x+h$, $\xi_{-h}(0)=x-h$ and $\xi(0)=x$. 
We define 
$$
T_{\xi_{\pm h}}(x\pm h)=\inf\{s: \xi_{\pm h}(-s)\in\partial \Omega,\ s\in[0,T(x)]\}.
$$
If $\xi_{\pm h}(s)\in \Omega$ for all $s\in[-T(x),0]$, then $T_{\xi_{\pm h}}(x\pm h)=+\infty$.
It is now convenient to distinguish three cases depending on which of $\xi$, $\xi_{\pm h}$ reaches $\partial \Omega$ first. 
\medskip

\noindent {\bf{Case 1}}: $T(x)=T_{\xi}(x)\leqslant\min\{T_{\xi_h}(x+h),T_{\xi_{-h}}(x-h)\}$ (see Figure 5.1).

\medskip
\noindent
Let $t^*=T_{\xi}(x)$ and set $x^+=\xi_h(-t^*)$, $x^-=\xi_{-h}(-t^*)$, and $\bar{x}=\bar{y}=\xi(-t^*)\in\partial \Omega$. Then, for any $y^{\pm}\in\partial\Omega$, we get
\begin{equation}\label{eq:semiconcavity1}
\begin{split}
	&u(x+h)+u(x-h)-2u(x)\\
	\leqslant& g(y^+)+\Phi^{\Omega}_L(y^+,x+h)+g(y^-)+\Phi^{\Omega}_L(y^-,x-h)-2g(\bar{x})-2A^{\Omega}_{T(x)}(\bar{x},x)\\
	\leqslant&\{g(y^+)+g(y^-)-2g(\bar{x})\}+\{\Phi^{\Omega}_L(y^+,x^+)+\Phi^{\Omega}_L(y^-,x^-)\}+c_1|h|^2.
\end{split}
\end{equation}
We also recall that, by the exterior sphere condition, there exists $K>0$ such that
\begin{equation}\label{eq:dist_semiconcave}
	\lambda d_{\partial\Omega}(x')+(1-\lambda)d_{\partial\Omega}(x'')-d_{\partial\Omega}(\lambda x'+(1-\lambda)x'')\leqslant\lambda(1-\lambda)K|x'-x''|^2
\end{equation}
for all $x',x''\in\overline{\Omega}$ and $\lambda\in[0,1]$. Since $\bar{x}\in\partial\Omega$ and $\bar{x}=(x^++x^-)/2$, then

\begin{align*}
	\frac 12 d_{\partial\Omega}(x^+)+\frac 12d_{\partial\Omega}(x^-)\leqslant K|h|^2,
\end{align*} 
and thus
$$
d_{\partial\Omega}(x^{\pm})\leqslant 2K|h|^2.
$$
Now, choose $y^{\pm}\in\partial\Omega$ in \eqref{eq:semiconcavity1} such that $|x^{\pm}-y^{\pm}|=d_{\partial\Omega}(x^\pm)\leqslant 2K|h|^2$. Then
\begin{equation}\label{eq:semiconcavity2}
	\Phi^{\Omega}_L(y^+,x^+)+\Phi^{\Omega}_L(y^-,x^-)\leqslant c_2|h|^2,
\end{equation}
where $c_2=4C_1K$ and $C_1$ is a Lipschitz constant of $\Phi^{\Omega}_L(\cdot,\cdot)$.
Moreover, the inequality
\begin{equation}\label{eq:semiconcavity3}
	g(y^+)+g(y^-)-2g(\bar{x})\leqslant c_3|h|^2
\end{equation}
follows from the estimates
\begin{align*}
	|y^+-y^-|\leqslant|y^+-x^+|+|x^+-x^-|+|x^--y^-|\leqslant& c_4|h|,\\
	|y^++y^--2\bar{x}|\leqslant|y^+-x^+|+|x^++x^--2\bar{x}|+|y^--x^-|\leqslant& c_5|h|^2,
\end{align*}
and \eqref{eq:g_semiconcave}. 
The combination of \eqref{eq:semiconcavity2}, \eqref{eq:semiconcavity3} and \eqref{eq:semiconcavity1} leads to our estimate.

\medskip

\begin{figure}
\begin{minipage}[t]{0.5\linewidth}
\begin{center}
	\begin{tikzpicture}
	\clip (-1.8,-.5) rectangle (2.5,2.5);
	\path [name path=arc1, draw=none] (-1.8,2.5) to[out=-45,in=-115] (1.8,2.5);
	\path [name path=arc0, draw=none] (0,0) to[out=120,in=60] (0,2.5);
	\path [name path=arc-, draw=none] (-1,0) to[out=120,in=60] (-1,2.5);
	\path [name path=arc+, draw=none] (1,0) to[out=120,in=60] (1,2.5);
	\draw [thick] (-1.8,2.5) to[out=-45,in=-115] (1.8,2.5);
	\draw [thick] (0,0) to[out=120,in=60] (0,2.5);
	\draw [thick] (-1,0) to[out=120,in=60] (-1,2.5);
	\draw [thick] (1,0) to[out=120,in=60] (1,2.5);
	
	\fill[gray!20] (-1.8,2.5) to[out=-45,in=-115] (1.8,2.5) -- (-1.8,2.5);
	\draw (1.3,2.1) node[above] {$\scriptscriptstyle \R^n\setminus\Omega$};
	\fill [name intersections={of = arc0 and arc1, by={b}}]
	(b) circle (1pt);
	\fill [name intersections={of = arc- and arc1, by={a}}]
	(a) circle (1pt);
	\fill [name intersections={of = arc+ and arc1, by={c}}]
	(c) circle (1pt);
	
	\fill [red] (0,0) circle (1pt);
	\draw (0,0) node[below] {$\scriptscriptstyle x$};
	\fill [red] (-1,0) circle (1pt);
	\draw (-1,0) node[below] {$\scriptscriptstyle x-h$};
	\fill [red] (1,0) circle (1pt);
	\draw (1,0) node[below] {$\scriptscriptstyle x+h$};
	
	\path [name path=a1, draw=none] (-2,1.67) -- (2,1.67);
	\fill [red][name intersections={of = a1 and arc-, by={A}}]
	(A) circle (1pt);
	\fill [red][name intersections={of = a1 and arc+, by={B}}]
	(B) circle (1pt);
	\fill [red][name intersections={of = a1 and arc0, by={C}}]
	(C) circle (1pt);
	\draw (A) node[left] {$\scriptscriptstyle x^-$};
	\draw (B) node[right] {$\scriptscriptstyle x^+$};
	\draw (a) node[above] {$\scriptscriptstyle y^-$};
	\draw (c) node[above] {$\scriptscriptstyle y^+$};
	\draw (b) node[above] {$\scriptscriptstyle \bar{x}=\bar{y}$};

	\draw [red] (-1,0) -- (0,0) -- (1,0);
	\draw [red] (b) --++ (1,0);
	\draw [red] (b) --++ (-1,0);
\end{tikzpicture}
\end{center}
\caption{fig1}
\label{fig:side:a}
\end{minipage}%
\begin{minipage}[t]{0.5\linewidth}
\begin{center}
	\begin{tikzpicture}
	\clip (-2.0,-1.7) rectangle (2.5,2.5);
	\path [name path=arc1, draw=none] (-2.5,-.8) to[out=25,in=-115] (.5,2.5);
	\path [name path=arc0, draw=none] (-1.5,-1.2) to[out=120,in=60] (-2,2.5);
	\path [name path=arc-, draw=none] (-.5,-1.2) to[out=120,in=60] (-1,2.5);
	\path [name path=arc+, draw=none] (.5,-1.2) to[out=120,in=60] (0,2.5);
	
	\draw [thick] (-2.5,-.8) to[out=25,in=-115] (.5,2.5);
	\draw [thick] (-1.5,-1.2) to[out=120,in=60] (-2,2.5);
	\draw [thick] (-0.5,-1.2) to[out=120,in=60] (-1,2.5);
	\draw [thick] (.5,-1.2) to[out=120,in=60] (0,2.5);
	
	\fill[gray!20] (-2.5,-.8) to[out=25,in=-115] (.5,2.5) -- (-2.5,2.5) -- (-2.5,-.8);
	\fill [name intersections={of = arc0 and arc1, by={b}}]
	(b) circle (1pt);
	\fill [name intersections={of = arc- and arc1, by={a}}]
	(a) circle (1pt);
	\fill [name intersections={of = arc+ and arc1, by={c}}]
	(c) circle (1pt);
	
	\fill [red]  (-0.5,-1.2) circle (1pt);
	\draw  (-0.5,-1.2) node[below] {$\scriptscriptstyle x$};
	\fill [red] (-1.5,-1.2) circle (1pt);
	\draw (-1.5,-1.2) node[below] {$\scriptscriptstyle x-h$};
	\fill [red] (.5,-1.2) circle (1pt);
	\draw (1,-1.2) node[below] {$\scriptscriptstyle x+h$};
	
	\path [name path=a1, draw=none] (-2,-0.38) -- (2,-0.38);
	\fill [red][name intersections={of = a1 and arc-, by={A}}]
	(A) circle (1pt);
	\fill [red][name intersections={of = a1 and arc0, by={B}}]
	(B) circle (1pt);
	\fill [red][name intersections={of = a1 and arc+, by={C}}]
	(C) circle (1pt);
	
	\fill [red] (.22,1.18) circle (1pt);
	\draw (.22,1.18) node[right] {$\scriptscriptstyle x^*$};
	
	\draw (-.75,-.5) node[right] {$\scriptscriptstyle \bar{x}$};
	\draw (C) node[right] {$\scriptscriptstyle x^+$};
	\draw (-1.75,-.5) node[right] {$\scriptscriptstyle x^-$};
	\draw (-1.65,-.4) node[above] {$\scriptscriptstyle y^-$};
	\draw (a) node[left] {$\scriptscriptstyle \bar{y}$};
	\draw (c) node[left] {$\scriptscriptstyle y^+$};
	\draw (-1.3,1.1) node[above] {$\scriptscriptstyle \R^n\setminus\Omega$};
	
	\draw [red] (B) -- (C);
	\draw [red] (-1.5,-1.2) -- (0.5,-1.2);
	\end{tikzpicture}
\end{center}
\caption{fig2}
\label{fig:side:b}
\end{minipage}
\end{figure}

\noindent {\bf{Case 2}}: $T_{\xi_{-h}}(x-h)<\min\{T_{\xi_h}(x+h),T_{\xi}(x)\}$ (or $T_{\xi_{h}}(x+h)<\min\{T_{\xi_{-h}}(x-h),T_{\xi}(x)\}$) (see Figure 5.2).

Without any loss of generality, we assume $T_{\xi_{-h}}(x-h)<\min\{T_{\xi_h}(x+h),T_{\xi}(x)\}$.
Let $t^*=T_{\xi_{-h}}(x-h)$,  $x^{\pm}=\xi_{\pm h}(-t^*)$ and $\bar{x}=\xi(-t^*)$. Similar to \eqref{eq:semiconcavity1}, we have that, for any $y^+\in\partial\Omega$,
\begin{equation}\label{eq:semiconcavity4}
\begin{split}
	&u(x+h)+u(x-h)-2u(x)\\
	\leqslant&\{g(x^-)+g(y^+)-2g(\bar{y})\}+\{\Phi^{\Omega}_L(y^+,x^+)-2\Phi^{\Omega}_L(\bar{y},\bar{x})\}+c_6|h|^2.
\end{split}
\end{equation}

Let $\eta(s)=\xi(s-t^*)$. Then $\eta(0)=\bar{x}$. For any $\tau\ll1$, we define a curve by
$$
\gamma_{\tau}(s)=\begin{cases}
	\eta(-\tau)-\eta(0)+\eta(s+\tau)+h, &s\in[-2\tau,-\tau],\\
	\eta(s)+h,& s\in[-\tau,0].
\end{cases}
$$
It is clear that for sufficiently small $\tau>0$, the arc $\gamma_{\tau}$ is contained in $\Omega$. We set
$$
\tau_0=\sup\{\tau>0: \text{the arc $\gamma_{\tau}$ is contained in $\Omega$} \}.
$$

Let
$$
\tau^*=T_{\eta}(\bar{x}).
$$
Then, for any projection of $\bar{x}$ to $\partial\Omega$, denoted by $z$, we have
\begin{align*}
g(\bar{y})+\Phi^\Omega_L(\bar{y},\bar{x})\leqslant g(z)+\Phi^\Omega_L(z,\bar{x})\leqslant g(z)+C_1d_{\partial \Omega}(\bar{x})\leqslant g(z)+C_1h,
\end{align*}
which yields that

\begin{align*}
\Phi^\Omega_L(\bar{y},\bar{x})\leqslant g(z)-g(\bar{y})+C_1h\leqslant \nu \Phi^\Omega_L(\bar{y},z)+C_1h\leqslant \nu \Phi^\Omega_L(\bar{y},\bar{x})+(1+\nu)C_1h.
\end{align*}
Thus, we deduce that
$$
\alpha \tau^*\leqslant\int_{-T(x)}^{-T(x)+\tau^*}L(\xi,\dot{\xi})\ ds=\Phi^\Omega_L(\bar{y},\bar{x})\leqslant \frac{1+\nu}{1-\nu}\cdot C_1h.
$$
Hence, we get 
$$
\tau^*\leqslant \frac1\alpha\cdot\frac{1+\nu}{1-\nu}\cdot C_1h.
$$

We have two cases which require a separate analysis by comparing $\tau_0$ and $\tau^*$.

\noindent {\bf{Case 2-1}}: If $\tau^*\leqslant\tau_0$, then $\gamma_{\tau^*}(s)\in\Omega$ for all $s\in(-2\tau^*,0]$.
In this case, we have that 
\begin{align*}	
		\gamma_{\tau^*}(0)&=x^+,\qquad \gamma_{\tau^*}(-\tau^*)=\eta(-\tau^*)+h=\bar{y}+h,\\
	x^*:=\gamma_{\tau^*}(-2\tau^*)&=2\eta(-\tau^*)-\eta(0)+h=2\bar{y}-(\bar{x}-h)=2\bar{y}-x^-.
\end{align*}
Thus we conclude that $\bar{y}=(x^{*}+x^-)/2\in\partial\Omega$, and $d_{\partial\Omega}(x^{*})\leqslant 2K|h|^2$ by \eqref{eq:dist_semiconcave}. It follows
\begin{align*}
	\Phi^{\Omega}_L(y^+,x^+)-2\Phi^{\Omega}_L(\bar{y},\bar{x})&\leqslant\Phi^{\Omega}_L(y^+,x^{*})+\Phi^{\Omega}_L(x^{*},x^+)-2\Phi^{\Omega}_L(\bar{y},\bar{x})\\
	&\leqslant\Phi^{\Omega}_L(y^+,x^{*})+\int^0_{-2\tau^{*}}L(\gamma_{\tau^*},\dot{\gamma}_{\tau^*})\ ds-2\int^0_{-\tau^{*}}L(\eta,\dot{\eta})\ ds\\
	&\leqslant\Phi^{\Omega}_L(y^+,x^{*})+c_7|h|^2.
\end{align*}
Taking $y^+$ to be any projection of $x^{*}$ to $\partial\Omega$, then
\begin{equation}\label{eq:semiconcavity5}
	\Phi^{\Omega}_L(y^+,x^+)-2\Phi^{\Omega}_L(\bar{y},\bar{x})\leqslant c_8|h|^2
\end{equation}
and 
\begin{equation}\label{eq:semiconcavity6}
	g(x^-)+g(y^+)-2g(\bar{y})\leqslant c_9|h|^2
\end{equation}
since $|x^--y^+|\leqslant|x^--x^{*}|+|x^{*}-y^+|\leqslant c_{10}|h|$ and $|x^-+y^+-2\bar{y}|\leqslant|x^-+x^{*}-2\bar{y}|+|x^{*}-y^+|\leqslant c_{11}|h|^2$. 
The desired estimate follows from the combination of \eqref{eq:semiconcavity4}, \eqref{eq:semiconcavity5} and \eqref{eq:semiconcavity6}.

\medskip

\noindent {\bf{Case 2-2}}: There exists $\tau\leqslant\tau^*$ such that $\gamma_{\tau}(-2\tau)\in\partial\Omega$.

\medskip

Let $x^{**}=\eta(-\tau)$ and $y^{**}=\gamma_{\tau}(-2\tau)\in\partial\Omega$. Then $x^{**}=(x^-+y^{**})/2$ and we get
\begin{equation*}\label{eq:semiconcavity7}
\begin{split}
	u(x+h)+u(x-h)-2u(x)
	&\leqslant g(x^-)+g(y^{**})-2g(\bar{y})-2\Phi^{\Omega}_L(\bar{y},x^{**})+c_{12}|h|^2\\
	&\leqslant g(x^-)+g(y^{**})-2g(\bar{y})+c_{12}|h|^2.
\end{split}
\end{equation*}
 Note that 
\begin{align*}
	g(x^-)+g(y^{**})-2g(\bar{y})&=g(x^-)+g(y^{**})-2G(x^{**})+2G(x^{**})-2g(\bar{y})\\
	&\leqslant c_{13}|h|^2+2\langle\nabla G(\bar{y}),x^{**}-\bar{y}\rangle+c_{14}|x^{**}-\bar{y}|^2\\
	&\leqslant c_{13}|h|^2+2\langle\nabla G(\bar{y}),\frac{x^--\bar{y}}{2}+\frac{y^{**}-\bar{y}}{2}\rangle+c_{14}|x^{**}-\bar{y}|^2\\
	&\leqslant c_{13}|h|^2+c_{15}|x^--\bar{y}|^2+c_{16}|y^{**}-\bar{y}|^2+c_{14}|x^{**}-\bar{y}|^2.
\end{align*}
Then from $|x^--\bar{y}|\leqslant|x^--\bar{x}|+|\bar{x}-\bar{y}|\leqslant c_{17}|h|$ and $|y^{**}-\bar{y}|\leqslant|\bar{y}-\bar{x}|+|\bar{x}-x^+|+|x^+-y^{**}|\leqslant c_{18}|h|$, we get
$$
g(x^-)+g(y^{**})-2g(\bar{y})\leqslant c_{19}|h|^2
$$
Thus, $u(x+h)+u(x-h)-2u(x)\leqslant c_{20}|h|^2$.

\medskip
\noindent {\bf{Case 3}}: $T_{\xi_{-h}}(x-h)=T_{\xi_h}(x+h)<T_{\xi}(x)$.
\medskip

Let $t^*=T_{\xi_{\pm h}}(x\pm h)$,  $x^{\pm}=\xi_{\pm h}(-t^*)\in\partial \Omega$ and $\bar{x}=\xi(-t^*)$. Similar to \eqref{eq:semiconcavity4}, we have that, for any $y^+\in\partial\Omega$,
\begin{align*}
	u(x+h)+u(x-h)-2u(x)
	\leqslant&\,g(x^-)+g(x^+)-2g(\bar{y})-2\Phi^{\Omega}_L(\bar{y},\bar{x})+c_{21}|h|^2.
\end{align*}
By similar arguments used in Case 2-2, we have $u(x+h)+u(x-h)-2u(x)\leqslant c_{22}|h|^2$.

The proof is complete.
\end{proof}

\subsection{Semiconvexity of $u$ near the boundary}


\begin{Lem}\label{lemmasemiconv}
Let $\Omega\subset\R^n$ be a bounded domain with $C^{1,1}$ boundary,
let $L$ be a Tonelli Lagrangian satisfying  $L\geqslant\alpha >0$ and let $g$ satisfy ({\bf G1}).
Then for any $\bar x\in\partial \Omega$ there exist $\eta$, $C>0$ such that: for any $x^+,x^-\in B(\bar x,\eta)\cap\overline{\Omega}$, $y^+,y^-\in\partial\Omega$, $T^+\geq T^->0$  and arcs $\xi^+\in\Gamma^{-T^+,0}_{y^+,x^+}(\overline{\Omega})$, $\xi^-\in\Gamma^{-T^-,0}_{y^-,x^-}(\overline{\Omega})$ satisfying
	$$
	u(x^\pm )=g(y^\pm)+\int^0_{-T^\pm}L(\xi^\pm(s),\dot{\xi^\pm}(s))\ ds,
	$$
we have that $\xi^\pm$ are of class $C^1$, $T^\pm\leqslant1$ and
\begin{equation}\label{estimlemma}
|\xi^+(s)-\xi^-(s)|+|\dot\xi^+(s)-\dot\xi^-(s)|\leqslant C|x^+-x^-|, \quad\forall\,s\in[-T^-,0],
\end{equation}
\begin{equation}\label{estimlemmaT}
|T^+-T^-|\leqslant C|x^+-x^-|.
\end{equation}
\end{Lem}

\begin{proof}
Fix $\bar x\in\partial\Omega$ and set $M=\max_{x\in \overline{B(\bar x,1)},|v|=1}L(x,v)$. 
In the rest of the proof, we use $c_i$ to denote certain positive constants, which depend only on $L$, $\Omega$ and $\bar{x}$.
Let $x^\pm, y^\pm, \xi^\pm, T^\pm$ be as in the statement, for $\eta\leqslant\frac{\alpha}{C'}\cdot\frac{1+\nu}{1-\nu}$ sufficiently small, where $C'=\max\{C_1,M\}$ and $C_1$ is the Lipschitz constant of $\Phi_L^\Omega$, that will be defined later in the proof. Consider $z^\pm\in\partial\Omega$ such that $|x^\pm-z^\pm|=d_{\partial\Omega}(x^\pm)$. Note that
\begin{equation}\label{etaT}
\begin{split}
g(y^\pm)+\Phi_L^\Omega(y^\pm,x^\pm)&= g(y^\pm)+\int^0_{-T^\pm}L(\xi^\pm(s),\dot{\xi^\pm}(s))\ ds\\
&\leqslant g(z^\pm)+\int^0_{-|x^\pm-z^\pm|}L\left(x^\pm+s\frac{x^\pm-z^\pm}{|x^\pm-z^\pm|},\frac{x^\pm-z^\pm}{|x^\pm-z^\pm|}\right)\ ds\\
&\leqslant g(z^\pm)+M|x^\pm-z^\pm|\leqslant g(z^\pm)+M\eta,
\end{split}
\end{equation}
which implies that
$$
\alpha T^\pm\leqslant\Phi_L^\Omega(y^\pm,x^\pm)\leqslant\frac{1+\nu}{1-\nu}\cdot C' \eta. 
$$
In particular, we can deduce that $T^\pm\leq1$. Further, recall that by Lemma \ref{upper_bound},
\begin{equation}\label{esssupdotxi}
\operatorname*{ess\ sup}_{s\in[-T^\pm,0]}|\dot{\xi}^\pm(s)|\leqslant C
\end{equation}
for some $C>0$. Therefore, $\xi^\pm([-T^\pm,0])\subset B(\bar x, 1+C)$.

Under our assumptions imposed on $L$ and $\partial \Omega$, if $\nu^\pm$ are the external unit normals to $\Omega$ at $y^\pm$, there exist (unique) $\mu^\pm>0$ satisfying
\begin{equation}\label{H=0}
H(y^\pm,-\mu^\pm\nu^\pm)=0.
\end{equation}
The Pontryagin's maximum principle ensures the existence of two arcs $p^\pm:[-T^\pm,0]\to\R^n$ satisfying
\begin{equation}\label{pmpham}
\left\{\begin{array}{ll}\dot\xi^\pm(s)=H_p(\xi^\pm(s),p^\pm(s)),\\
\dot p^\pm(s)=-H_x(\xi^\pm(s),p^\pm(s)),\end{array}\right.
\end{equation}
\begin{equation}\label{maxpr}
\langle p^\pm(s),\dot\xi^\pm(s)\rangle-L(\xi^\pm(s),\dot\xi^\pm(s))=\max_{|v|\leqslant 2C}\left[\langle p^\pm(s),v\rangle-L(\xi^\pm(s),v)\right],
\end{equation}
\begin{equation}\label{superd}
p^\pm(s)\in D^+u(\xi^\pm(s)),\quad\forall\,s\in(-T^\pm,0],
\end{equation}
and
\begin{equation}\label{transv}
p^\pm(-T^\pm)=-\mu^\pm\nu^\pm.
\end{equation}
As a consequence, \eqref{pmpham} implies that $\xi^\pm$ and $p^\pm$ are of class $C^1$ and \eqref{maxpr} implies that
\begin{equation}\label{carp}
p^\pm(s)=L_v(\xi^\pm(s),\dot\xi^\pm(s)),\quad\forall\,s\in(-T^\pm,0).
\end{equation}
Hence,
\begin{equation}\label{boundmu}
|p^\pm(s)|\leqslant\max_{x\in \overline{B(\bar x,1+C)},|v|\leq C}|L_v(x,v)|,\quad\forall\,s\in[-T^\pm,0].
\end{equation}

\medskip
In order to prove \eqref{estimlemma}-\eqref{estimlemmaT}, we proceed in several steps.
\medskip

\noindent {\bf{Step 1}}:  Since $\partial\Omega$ is $C^{1,1}$, \eqref{H=0} and \eqref{boundmu}, we obtain
\begin{equation*}\begin{split}
0&=H(y^-,-\mu^-\nu^-)\geqslant-c_1|y^+-y^-|+H(y^+,-\mu^-\nu^-)\\
&\geqslant-c_1|y^+-y^-|+\langle -\mu^-\nu^-,\dot\xi^+(-T^+)\rangle-L(y^+,\dot\xi^+(-T^+))\\
&=-c_1|y^+-y^-|+\langle -\mu^-\nu^-+\mu^+\nu^+,\dot\xi^+(-T^+)\rangle+H(y^+,-\mu^+\nu^+)\\
&=-c_1|y^+-y^-|+\langle -\mu^-(\nu^--\nu^+),\dot\xi^+(-T^+)\rangle+\langle -(\mu^--\mu^+)\nu^+,\dot\xi^+(-T^+)\rangle\\
&\geqslant-c_2|y^+-y^-|+(\mu^--\mu^+)\langle -\nu^+,\dot\xi^+(-T^+)\rangle.
\end{split}
\end{equation*}
Since $\langle -\mu^+\nu^+,\dot\xi^+(-T^+)\rangle=L(y^+,\dot\xi^+(-T^+))\geq\alpha>0$ and $\mu^+$ is bounded by \eqref{boundmu}, we deduce that $\mu^--\mu^+\leq c_3|y^+-y^-|$. By exchanging + and -, we finally obtain
\begin{equation}\label{distmu}
|\mu^--\mu^+|\leqslant c_3|y^+-y^-|.
\end{equation}

\noindent {\bf{Step 2}}: Define a curve $\xi:\big[-|\xi^+(-T^-)-y^-|,0\big]\to \R^n$ by
$$s\mapsto\xi^+(-T^-)+s\frac{\xi^+(-T^-)-y^-}{|\xi^+(-T^-)-y^-|}$$
and set $t^*=\inf\{t\in(0,|\xi^+(-T^-)-y^-|]:\xi(-t)\in\partial\Omega\}$. Observe that
\begin{equation*}
\begin{split}
g(y^+)&+\Phi_L^\Omega(y^+,\xi^+(-T^-))= g(y^+)+\int^{-T^-}_{-T^+}L(\xi^+(s),\dot{\xi^+}(s))\ ds\\
&\leqslant g(\xi(-t^*))+\int^0_{-t^*}L\left(\xi(s),\frac{\xi^+(-T^-)-y^-}{|\xi^+(-T^-)-y^-|}\right)\ ds\\
&\leqslant g(\xi(-t^*))+M|\xi^+(-T^-)-y^-|,
\end{split}
\end{equation*}
which implies that
$$
\alpha(T^+-T^-)\leqslant\int_{-T^+}^{-T^-}L(\xi^+,\dot{\xi}^+)\ ds\leqslant\frac{1+\nu}{1-\nu}\cdot C'\cdot |\xi^+(-T^-)-y^-|.
$$
Consequently,
\begin{equation}\label{etaT2}
T^+-T^-\leqslant \frac{1}{\alpha}\cdot\frac{1+\nu}{1-\nu} \cdot C'\cdot|\xi^+(-T^-)-y^-|.
\end{equation}

By \eqref{pmpham}, \eqref{transv}, \eqref{boundmu}, \eqref{distmu}, \eqref{etaT2}, we have
\begin{equation}\label{p-T}
\begin{split}
|p^+(-T^-)-p^-(-T^-)|&\leqslant|p^+(-T^-)-p^+(-T^+)|+|\mu^+\nu^+-\mu^-\nu^-|\\
&\leqslant c_4\Big((T^+-T^-)+|y^+-y^-|\Big)\\
&\leqslant c_5|\xi^+(-T^-)-y^-|+c_4|y^+-\xi^+(-T^-)|\\
&\leqslant c_5|\xi^+(-T^-)-y^-|+c_6(T^+-T^-)\\
&\leqslant c_7|\xi^+(-T^-)-y^-|.
\end{split}
\end{equation}

\noindent {\bf{Step 3}}:  Set $H_p^\pm=H_p(\xi^\pm,p^\pm)$ and $H_x^\pm=H_x(\xi^\pm,p^\pm)$.
By \eqref{pmpham} and the regularity assumptions on $H$, we have for every $s\in(-T^-,0)$ , we have
\begin{equation}\label{1ineq}
\begin{split}
&\frac{d}{ds}\langle\xi^+-\xi^-,p^+-p^-\rangle\\
&=\langle\dot\xi^+-\dot\xi^-,p^+-p^-\rangle+\langle\xi^+-\xi^-,\dot p^+-\dot p^-\rangle\\
&=\langle H_p^+-H_p^-,p^+-p^-\rangle-\langle\xi^+-\xi^-,H_x^+-H_x^-\rangle\\
&\geqslant\langle H_p^+-H_p(\xi^+,p^-),p^+-p^-\rangle+\langle H_p(\xi^+,p^-)-H_p^-,p^+-p^-\rangle\\
&\qquad-c_8|\xi^+-\xi^-|(|\xi^+-\xi^-|+|p^+-p^-|)\\
&\geqslant c_9|p^+-p^-|^2-c_{10}|\xi^+-\xi^-||p^+-p^-|-c_{8}|\xi^+-\xi^-|^2\\
&\geqslant c_{11}|p^+-p^-|^2-c_{12}|\xi^+-\xi^-|^2.
\end{split}
\end{equation}
On the other hand, by  \eqref{superd}, \eqref{transv}, \eqref{boundmu}, \eqref{distmu} and Proposition \ref{global_semi_value}, we get
\begin{equation}\label{2ineq}
\begin{split}
&\,\int_{-T^-}^0\frac{d}{ds}\langle\xi^+-\xi^-,p^+-p^-\rangle\ ds\\
=&\,\langle x^+-x^-,p^+(0)-p^-(0)\rangle-\langle \xi^+(-T^-)-y^-, -\mu^+\nu^++\mu^-\nu^-\rangle\\
\leqslant&\,c_{13}|x^+-x^-|^2+\langle \xi^+(-T^-)-y^-,\mu^+(\nu^+-\nu^-)\rangle\\
&\qquad+\langle \xi^+(-T^-)-y^-,(\mu^+-\mu^-)\nu^-\rangle\\
\leqslant&\,c_{13}|x^+-x^-|^2+ c_{14}|\xi^+(-T^-)-y^-|^2.
\end{split}
\end{equation}
Combining \eqref{1ineq} and \eqref{2ineq}, we obtain
\begin{equation}\label{3ineq}
\begin{split}
	&\,\int_{-T^-}^0|p^+-p^-|^2\ ds\\
	\leqslant&\,c_{15}\left(|x^+-x^-|^2+ |\xi^+(-T^-)-y^-|^2+\int_{-T^-}^0|\xi^+-\xi^-|^2\ ds\right).
\end{split}
\end{equation}
Let $0<\varepsilon\leqslant1/4ec_{15}$. Then there exists $c_{17}>0$ such that
\begin{equation*}
\begin{split}
\frac{1}{2}\frac{d}{ds}|\xi^+-\xi^-|^2&=\langle\dot\xi^+-\dot\xi^-,\xi^+-\xi^-\rangle=\langle H_p^+-H_p^-,\xi^+-\xi^-\rangle\\
&\geqslant -c_{16}\left(|\xi^+-\xi^-|^2+|\xi^+-\xi^-||p^+-p^-|\right)\\
&\geqslant -\frac{c_{17}}{2}|\xi^+-\xi^-|^2-\frac{\varepsilon}{2}|p^+-p^-|^2.
\end{split}
\end{equation*}
Then we obtain
\begin{equation}\label{estxi}
\begin{split}
	&\,|\xi^+(s)-\xi^-(s)|^2\\
	\leqslant&\,\left(|x^+-x^-|^2+\varepsilon\int_{-T^-}^0|p^+(r)-p^-(r)|^2\ dr\right)e^{c_{17}T^-}
\end{split}
\end{equation}
for every $s\in[-T^-,0]$.
By \eqref{3ineq} and \eqref{estxi} we have
\begin{align*}
	&\,\int_{-T^-}^0|p^+-p^-|^2\ ds\\
	\leqslant&\,c_{15}\left[|x^+-x^-|^2+2\left(|x^+-x^-|^2+\varepsilon\int_{-T^-}^0|p^+-p^-|^2\ ds\right)e^{c_{17}T^-}\right].
\end{align*}
By choosing $\eta$ sufficiently small and recalling \eqref{etaT}, we can assume $c_{17}T^-\leqslant1$. From our choice of $\varepsilon$, we can conclude that
\begin{equation}\label{lastp}
\int_{-T^-}^0|p^+-p^-|^2\ ds\leqslant c_{18}|x^+-x^-|^2.
\end{equation}
Finally, replacing \eqref{lastp} in \eqref{estxi}, we obtain
\begin{equation}\label{lastb}
|\xi^+(s)-\xi^-(s)|^2\leqslant c_{19}|x^+-x^-|^2.
\end{equation}
By \eqref{etaT2} and \eqref{lastb}, we deduce \eqref{estimlemmaT}.

\noindent {\bf{Step 4}}: Analogously, since
\begin{equation*}
\begin{split}
\frac{1}{2}\frac{d}{ds}|p^+-p^-|^2&=\langle\dot p^+-\dot p^-,p^+-p^-\rangle=\langle -H_x^++H_x^-,p^+-p^-\rangle\\
&\leqslant c_{20}\left(|\xi^+-\xi^-||p^+-p^-|+|p^+-p^-|^2\right)\\
&\leqslant \frac{1}{2}|\xi^+-\xi^-|^2+\frac{c_{21}}{2}|p^+-p^-|^2,
\end{split}
\end{equation*}
by Gronwall's inequality, we obtain
\begin{equation}\label{estp}
\begin{split}
	&\,|p^+(s)-p^-(s)|^2\\
	\leqslant&\,\left(|p^+(-T^-)-p^-(-T^-)|^2+\int_{-T^-}^0|\xi^+(r)-\xi^-(r)|^2\ dr\right)e^{c_{21}T^-}
\end{split}
\end{equation}
for every $s\in[-T^-,0]$.
By \eqref{p-T}, \eqref{lastb} and \eqref{estp}, we deduce that
$$|p^+(s)-p^-(s)|^2\leqslant c_{22}|x^+-x^-|^2.$$
Consequently,
\begin{equation}\label{distdotxi}
|\dot\xi^+-\dot\xi^-|=|H_p^+-H_p^-|\leqslant c_{23}\left(|\xi^+-\xi^-|+|p^+-p^-|\right)\leqslant c_{24}|x^+-x^-|.
\end{equation}

The proof is  complete.
\end{proof}

\begin{Pro}\label{semiconvex}
Let $g$ satisfy ({\bf G1}) and ({\bf G2}).
Let $\Omega\subset\R^n$ be a bounded domain,
and let $L$ be a Tonelli Lagrangian, which also satisfy
\begin{equation}\label{hppmp2}
\text{$\Omega$ has $C^2$ boundary and $L\geqslant\alpha>0$.}
\end{equation}
Then for any $\bar x\in\partial \Omega$ there exist $\tilde{\eta}$, $\tilde{C}>0$ such that $u$ is semiconvex with constant $\tilde{C}$ on $B(\bar x,\tilde{\eta})\cap\overline{\Omega}$.
\end{Pro}

\begin{proof}
Throughout this proof we use $C$ to denote a generic positive constant not necessarily the same in any two places.
Fix $\bar x\in\partial\Omega$ and let $\eta$ be as in Lemma \ref{lemmasemiconv}. 
By assumption \eqref{hppmp2}, there exist $\delta$, $c'>0$ such that for any $y^+,y^-\in B(\bar x,\delta)\cap\partial\Omega$, we have
\begin{equation}\label{distdelta}
d_{\partial\Omega}\left(\frac{y^++y^-}{2}\right)\leqslant c'|y^+-y^-|^2.
\end{equation}
We shall prove the existence of $\tilde{\eta},\tilde{C}>0$ such that
$$
\Delta:=2u(x)-u(x+h)-u(x-h)\leqslant \tilde{C}|h|^2,
$$
for all $x,h\in\R^n$ satisfying $[x-h,x+h]\in B(\bar x,\tilde{\eta})\cap\overline{\Omega}$.
Fix such $x$ and $h$. By Corollary \ref{exit_time_prob}, there exist $y^+,y^-\in\partial\Omega$, $T^+, T^->0$  and arcs  $\xi^+\in\Gamma^{-T^+,0}_{y^+,x+h}(\overline{\Omega})$, $\xi^-\in\Gamma^{-T^-,0}_{y^-,x-h}(\overline{\Omega})$ such that
	$$
	u(x^\pm )=g(y^\pm)+\int^0_{-T^\pm}L(\xi^\pm(s),\dot{\xi^\pm}(s))\ ds,
	$$
	where $x^\pm=x\pm h$.
Remarking as in \eqref{etaT} that $T^\pm\leqslant M\tilde{\eta}/\alpha$ and recalling \eqref{esssupdotxi}, we obtain $\xi^\pm([-T^\pm,0])\subset B(\bar x, \tilde{\eta}+CM\tilde{\eta}/\alpha)$. Choose $\tilde{\eta}\leqslant\eta$ such that $\xi^\pm([-T^\pm,0])\subset B(\bar x, \delta)$.

We can suppose without loss of generality that $T^-\leqslant T^+$.	
Set $T^*=\frac{T^++T^-}{2}$ and let $\xi:[-T^*,0]\to\R^n$ be the absolutely continuous arc satisfying $\xi(0)=x$ defined by
$$
\xi(s)=\left\{\begin{array}{llll}\dfrac{\xi^+(s)+\xi^-(s)}{2}, & s\in[-T^-,0],\\
\dfrac{\xi^+(s)+\xi^+(s+T^--T^*)-\xi^+(-T^*)+\xi^-(-T^-)}{2}, & s\in[-T^*,-T^-).\end{array}\right.
$$
It is now convenient to distinguish two cases.
\medskip

\noindent {\bf{Case 1}}: $\xi([-T^*,0])\cap\partial\Omega\neq\emptyset$. Call $\tau^*=\inf\{s\in[0,T^*]:\xi(-s)\in\partial\Omega\}$.
\medskip

If $\tau^*\leqslant T^-$, by the regularity of $L$ and Lemma \ref{lemmasemiconv} we have
$
\Delta\leqslant A+B,
$
where
$
A:=2g(\xi(-\tau^*))-g(y^+)-g(y^-)$
and
$$
B:=2\int_{-\tau^*}^0L(\xi(s),\dot\xi(s))\ ds
-\int_{-T^+}^0L(\xi^+(s),\dot\xi^+(s))\ ds-\int_{-T^-}^0L(\xi^-(s),\dot\xi^-(s))\ ds.
$$
It is easy to see that
	\begin{align*}	
	B\leqslant&\,
	\int_{-\tau^*}^0\left(2L\Big(\dfrac{\xi^+(s)+\xi^-(s)}{2},\dfrac{\dot\xi^+(s)+\dot\xi^-(s)}{2}\Big)-L(\xi^+(s),\dot{\xi}^+(s))-L(\xi^-(s),\dot{\xi}^-(s))\right)\ ds\\
	\leqslant&\,C\int_{-\tau^*}^0\left(|\xi^+(s)-\xi^-(s)|^2+|\dot\xi^+(s)-\dot\xi^-(s)|^2\right)\ ds\\
	\leqslant&\, C|h|^2.
\end{align*}
Now we estimate $A$ as follows.
\begin{align*}
	A&=2G(\xi(-T^*))-g(y^+)-g(y^-)+2g(\xi(-\tau^*))-2G(\xi(-T^*))\\
	&\leqslant C|y^--y^+|^2+2g(\xi(-\tau^*))-2G(\xi(-T^*)).
\end{align*}
In view of Lemma \ref{lemmasemiconv}, we have
$$
|y^--y^+|\leqslant |y^--\xi^+(-T^-)|+|\xi^+(-T^-)-y^+|\leqslant C|h|.
$$
Set $\Sigma=2g(\xi(-\tau^*))-2G(\xi(-T^*))$. By ({\bf G2}), we have
\begin{align*}
	\Sigma &\leqslant 2\langle\nabla G(\xi(-\tau^*)),\xi(-\tau^*)-\xi(-T^*)\rangle+C|\xi(-\tau^*)-\xi(-T^*)|^2\\
	&=2\Big\langle\nabla G(\xi(-\tau^*)),(\frac{\xi(-\tau^*)-y^-}{2}+\frac{\xi(-\tau^*)-y^+}{2})\Big\rangle+C|\xi(-\tau^*)-\xi(-T^*)|^2\\
	&\leqslant C|\xi(-\tau^*)-y^-|^2+C|\xi(-\tau^*)-y^+|^2+C|\xi(-\tau^*)-\xi(-T^*)|^2\\
	&=:a_1+a_2+a_3,
\end{align*}
where $a_1=C|\xi(-\tau^*)-y^-|^2$, $a_2=C|\xi(-\tau^*)-y^+|^2$ and $a_3=C|\xi(-\tau^*)-\xi(-T^*)|^2$.

We now prove that $a_1\leqslant C|h|^2$. By similar arguments one can show that $a_2\leqslant C|h|^2$ and 
$a_3\leqslant C|h|^2$ follows from that above two inequalities.  Since
\begin{align*}
	|\xi(-\tau^*)-y^-|&\leqslant \frac{|\xi^-(-\tau^*)-y^-|}{2}+\frac{|\xi^+(-\tau^*)-y^-|}{2}\\
	& \leqslant|\xi^-(-\tau^*)-y^-|+\frac{|\xi^-(-\tau^*)-\xi^+(-\tau^*)|}{2}\\
	&\leqslant C|T^--\tau^*|+C|h|,
\end{align*}
then if $|T^--\tau^*|\leqslant C|h|$, we have 
$a_1\leqslant C|h|^2$. So, it suffices to show that $|T^--\tau^*|\leqslant C|h|$.
Note that 
\begin{align*}
	g(y^-)+\Phi_L^\Omega(y^-,\xi^-(-\tau^*))\leqslant g(z)+\Phi_L^\Omega(z,\xi^-(-\tau^*))\leqslant g(z)+Cd_{\partial\Omega}(\xi^-(-\tau^*))\leqslant g(z)+C|h|,
\end{align*}
where $z$ is an arbitrary projection of $\xi^-(-\tau^*)$ on $\partial\Omega$. Then we get
$$
\Phi_L^\Omega(y^-,\xi^-(-\tau^*))\leqslant \frac{1+\nu}{1-\nu}\cdot C|h|.
$$
Thus we have
$$
T^--\tau^*\leqslant \frac{1}{\alpha}\cdot \frac{1+\nu}{1-\nu}\cdot C|h|.
$$
So far, we have shown that $A\leqslant C|h|^2$ and thus $\Delta\leqslant C|h|^2$.

If $T^-<\tau^*\leqslant T^*$, then we have
\begin{align*}
\Delta &\,\leqslant2g(\xi(-\tau^*))-g(y^+)-g(y^-)\\
	&\,+\int_{-T^-}^0\left(2L(\xi(s),\dot\xi(s))-L(\xi^+(s),\dot{\xi}^+(s))-L(\xi^-(s),\dot{\xi}^-(s))\right)\ ds\\
	&\,+2\int_{-\tau^*}^{-T^-}L(\xi(s),\dot\xi(s))\ ds-\int_{-T^*}^{-T^-}L(\xi^+(s),\dot\xi^+(s))\ ds-\int_{-T^+}^{-T^*}L(\xi^+(s),\dot\xi^+(s))\ ds.
\end{align*}
By arguments analogous to the ones used in the previous case, one can show 
$$
2g(\xi(-\tau^*))-g(y^+)-g(y^-)\leqslant C|h|^2.
$$
The convexity of $L$ with respect to the variable $v$ yields
\begin{align*}
	&\int_{-T^-}^0\left(2L(\xi(s),\dot\xi(s))-L(\xi^+(s),\dot{\xi}^+(s))-L(\xi^-(s),\dot{\xi}^-(s))\right)\ ds\\
	&\,+2\int_{-\tau^*}^{-T^-}L(\xi(s),\dot\xi(s))\ ds-\int_{-T^*}^{-T^-}L(\xi^+(s),\dot\xi^+(s))\ ds-\int_{-T^+}^{-T^*}L(\xi^+(s),\dot\xi^+(s))\ ds\\
	\leqslant&\,C |h|^2+2\int_{-\tau^*}^{-T^-}L\Big(\xi(s),\dfrac{\dot\xi^+(s)+\dot\xi^+(s+T^--T^*)}{2}\Big)\ ds\\
	\qquad&-\int_{-\tau^*}^{-T^-}L(\xi^+(s),\dot\xi^+(s))\ ds-\int_{-T^*-(\tau^*-T^-)}^{-T^*}L(\xi^+(s),\dot\xi^+(s))\ ds\\
	\leqslant&\,C |h|^2+\int_{-\tau^*}^{-T^-}\Big(L(\xi(s),\dot\xi^+(s))+L(\xi(s),\dot\xi^+(s+T^--T^*))\Big)\ ds\\
	\qquad&-\int_{-\tau^*}^{-T^-}L(\xi^+(s),\dot\xi^+(s))\ ds-\int_{-\tau^*}^{-T^-}L(\xi^+(s+T^--T^*),\dot\xi^+(s+T^--T^*))\ ds.
\end{align*}
Therefore, we get
\begin{equation}\label{case1b}
\begin{split}
	\Delta\leqslant&\,C |h|^2+\int_{-\tau^*}^{-T^-}\left|L(\xi(s),\dot\xi^+(s))-L(\xi^+(s),\dot\xi^+(s))\right|\ ds\\
	\qquad&+\int_{-\tau^*}^{-T^-}\left|L(\xi(s),\dot\xi^+(s+T^--T^*))-L(\xi^+(s+T^--T^*),\dot\xi^+(s+T^--T^*))\right|\ ds\\
	\leqslant&\,C |h|^2+C\int_{-\tau^*}^{-T^-}\Big(\left|\xi(s)-\xi^+(s)\right|+\left|\xi(s)-\xi^+(s+T^--T^*)\right|\Big)\ ds.
\end{split}
\end{equation}
Observe that for $s\in[-T^*,-T^-)$ we have
\begin{equation*}
\begin{split}
	\xi(s)&=\xi(-T^-)+\dfrac{1}{2}\left(\xi^+(s)-\xi^+(-T^-)\right)+\dfrac{1}{2}\left(\xi^+(s+T^--T^*)-\xi^+(-T^*)\right)\\
	&=\dfrac{1}{2}\left(y^-+\xi^+(s)+\xi^+(s+T^--T^*)-\xi^+(-T^*)\right).
\end{split}
\end{equation*}
Therefore, by \eqref{estimlemma}-\eqref{estimlemmaT} we deduce
\begin{equation}\label{perc1}
\begin{split}
	&\,|\xi(s)-\xi^+(s)|\\
	\leqslant&\,\dfrac{1}{2}\Big(|y^--\xi^+(s)|+|\xi^+(s+T^--T^*)-\xi^+(-T^*)|\Big)\\
	\leqslant&\,\dfrac{1}{2}\Big(|y^--\xi^+(-T^-)|+|\xi^+(-T^-)-\xi^+(s)|+C\,\dfrac{T^+-T^-}{2}\Big)\\
	\leqslant&\,C|h|
\end{split}
\end{equation}
and
\begin{equation}\label{perc2}
\begin{split}
	&\,|\xi(s)-\xi^+(s+T^--T^*)|\\
	&\,\leqslant\dfrac{1}{2}\Big(|y^--\xi^+(-T^*)|+|\xi^+(s)-\xi^+(s+T^--T^*)|\Big)\\
	&\,\leqslant\dfrac{1}{2}\Big(|y^--\xi^+(-T^-)|+|\xi^+(-T^-)-\xi^+(-T^*)|+C\,\dfrac{T^+-T^-}{2}\Big)\\
	&\,\leqslant C|h|.
\end{split}
\end{equation}
Finally, by \eqref{case1b}, \eqref{perc1} and \eqref{perc2} we obtain
\begin{equation*}
2u(x)-u(x+h)-u(x-h)\leqslant C |h|^2+C(\tau^*-T^-)|h|\leqslant C |h|^2.
\end{equation*}

\medskip

\noindent {\bf{Case 2}}: $\xi([-T^*,0])\cap\partial\Omega=\emptyset$.

\medskip
\noindent Since $\xi(-T^*)=\frac{y^-+y^+}{2}$, by \eqref{distdelta} there exists $y^*\in \partial\Omega$ such that
$$|\xi(-T^*)-y^*|\leqslant c'|y^+-y^-|^2\leqslant C|h|^2.$$
Thus,
\begin{equation*}
\begin{split}
	\Delta\leqslant&\,2g(y^*)-g(y^+)-g(y^-)\\
	&\,+\int_{-T^-}^0\left(2L(\xi(s),\dot\xi(s))-L(\xi^+(s),\dot{\xi}^+(s))-L(\xi^-(s),\dot{\xi}^-(s))\right)\ ds\\
	&\,+2\int_{-T^*}^{-T^-}L(\xi,\dot{\xi})\ ds-\int_{-T^+}^{-T^-}L(\xi^+,\dot{\xi}^+)\ ds\\
	&\,+2\int^0_{-|\xi(-T^*)-y^*|}L\left(\xi(-T^*)+s\frac{\xi(-T^*)-y^*}{|\xi(-T^*)-y^*|},\frac{\xi(-T^*)-y^*}{|\xi(-T^*)-y^*|}\right)\ ds\\
	\leqslant&\,2g(y^*)-g(y^+)-g(y^-)+C |h|^2+2M|\xi(-T^*)-y^*|\\
	&+2\int_{-T^*}^{-T^-}L(\xi,\dot{\xi})\ ds-\int_{-T^+}^{-T^-}L(\xi^+,\dot{\xi}^+)\ ds\\
	\leqslant&\, 2g(y^*)-g(y^+)-g(y^-)+ C |h|^2+2\int_{-T^*}^{-T^-}L(\xi,\dot{\xi})\ ds-\int_{-T^+}^{-T^-}L(\xi^+,\dot{\xi}^+)\ ds.
\end{split}
\end{equation*}
By similar arguments used in Case 1, one can obtain that
$|\xi(s)-\xi^+(s)|\leqslant C|h|$ and $|\xi(s)-\xi^+(s+T^--T^*)|\leqslant C|h|$ for $s\in[-T^*,-T^-]$.
Thus, we have
\begin{align*}
	&\,2\int_{-T^*}^{-T^-}L(\xi,\dot{\xi})\ ds-\int_{-T^+}^{-T^-}L(\xi^+,\dot{\xi}^+)\ ds\\
	&\,\leqslant \int_{-T^*}^{-T^-}L(\xi,\dot{\xi}^+)+L(\xi,\dot{\xi}^+(s+T^--T^*))\ ds-\int_{-T^+}^{-T^*}L(\xi^+,\dot{\xi}^+)\ ds-\int_{-T^*}^{-T^-}L(\xi^+,\dot{\xi}^+)\ ds\\
	&\, \leqslant C|h|^2.
\end{align*}
Note that 
\begin{align*}
	2g(y^*)-g(y^+)-g(y^-)&=2G(\xi(-T^*))-g(y^+)-g(y^-)+2g(y^*)-2G(\xi(-T^*))\\
	&\leqslant C|y^--y^+|^2+2g(y^*)-2G(\xi(-T^*))\\
	&\leqslant C|h|^2+2g(y^*)-2G(\xi(-T^*)).
\end{align*}
Using ({\bf G2}), we have
\begin{align*}
	2g(y^*)-2G(\xi(-T^*))&\leqslant 2\langle\nabla G(y^*),y^*-\xi(-T^*)\rangle+C|y^*-\xi(-T^*)|^2\\
	&=2\Big\langle\nabla G(y^*),(\frac{y^*-y^-}{2}+\frac{y^*-y^+}{2})\Big\rangle+C|y^*-\xi(-T^*)|^2\\
	&\leqslant C|y^*-y^-|^2+C|y^*-y^+|^2+C|y^*-\xi(-T^*)|^2\\
	&=C|h|^2,
\end{align*}
which implies that $\Delta\leqslant C|h|^2$. 

The proof is complete.
\end{proof}

\vskip.3cm

{\appendix

\section{Properties of relative Ma\~n\'e's potentials}
\label{sse:a2}

\begin{proof}[Proof of Lemma \ref{poten_proper}]
The triangle inequality in (1) can be obtained directly by the definition.

To prove (2), first we have $A^{\Omega}_t(x,x)\leqslant|L(x,0)|t$ for any $x\in\overline{\Omega}$ and $t>0$. Thus $\Phi^{\Omega}_L(x,x)\leqslant0$. If $\Phi^{\Omega}_L(x,x)<0$, then there exists $T>0$ such that $A^{\Omega}_T(x,x)<0$. It follows $-\inf_{t>0}\frac 1tA^{\Omega}_t(x,x)=c_{\Omega}(L)>0$ which leads to a contradiction.

Now we turn to the proof of (3). For any $x$, $y\in\overline{\Omega}$, since $\Omega$ is $C$-quasiconvex (see Definition \ref{x_C_reachable} and Remark \ref{quasiconvex}), then there exists a Lipschitz curve $\gamma\in\Gamma^{0,t(x,y)}_{x,y}(\overline{\Omega})$ with $|\dot{\gamma}|=1$ and $t(x,y)\leqslant C|x-y|$. Thus, we have that
\begin{equation}\label{eq:ap_Mane}
	\Phi^{\Omega}_L(x,y)\leqslant\int_0^{t(x,y)}L(\gamma,\dot{\gamma})ds\leqslant\theta_2(1)t(x,y)\leqslant \theta_2(1)C|x-y|. 
\end{equation}
Due to (1) and (2), we have that
$$
\Phi^{\Omega}_L(x,y)+\Phi^{\Omega}_L(y,x)\geqslant\Phi^{\Omega}_L(x,x)=0.
$$
Therefore, from \eqref{eq:ap_Mane}, we obtain
$$
\Phi^{\Omega}_L(x,y)\geqslant-\Phi^{\Omega}_L(y,x)\geqslant-\theta_2(1)C|x-y|.
$$
This completes the proof of (3).

Finally, by (1) it is easy to see that 
\begin{align*}
	\Phi^{\Omega}_L(x_1,y_1)-\Phi^{\Omega}_L(x_2,y_2)\leqslant& \Phi^{\Omega}_L(x_1,x_2)+\Phi^{\Omega}_L(y_2,y_1)\\
	\leqslant&\theta_2(1)C(|x_1-x_2|+|y_1-y_2|).
\end{align*}
Changing the roles of $(x_1,y_1)$ and $(x_2,y_2)$ we obtain (4).
\end{proof}

\begin{proof}[Proof of Lemma \ref{eq:equiv_Phi_0}]
	In order to show the equivalence of (1) and (2), we only need to prove that, if there exist $x_0$, $y_0\in\overline{\Omega}$ such that $\Phi^{\Omega}_L(x_0,y_0)>-\infty$, then $\Phi^{\Omega}_L(x,y)>-\infty$ for all $x,y\in\overline{\Omega}$. Otherwise, there would be 
	$x_1$, $y_1\in\overline{\Omega}$ such that $\Phi^{\Omega}_L(x_1,y_1)=-\infty$. By Lemma \ref{poten_proper} (1), we get
	\[
	\Phi^{\Omega}_L(x_1,x_1)\leqslant \Phi^{\Omega}_L(x_1,y_1)+\Phi^{\Omega}_L(y_1,x_1)=-\infty.
	\]
	Thus, there is a closed curve $\gamma_1:[0,T_1]\to\overline{\Omega}$ with $\gamma_1(0)=\gamma_1(T_1)=x_1$ such that
	
	\[
	\int_0^{T_1}L(\gamma_1,\dot{\gamma}_1)\ ds<0.
	\]
	By going around $\gamma_1$ many times, it is clear that
$\Phi^{\Omega}_L(x,y)=-\infty$ for all $x$, $y\in\overline{\Omega}$, a contradiction.

Next, we show the equivalence of (2) and (3). 
If $c_{\Omega}(L)>0$, then by definition, we have
\[
\inf_{t>0,x\in\overline{\Omega}}\frac 1tA^{\Omega}_t(x,x)<0,
\]
which implies there is a closed curve $\gamma_2:[0,T_2]\to\overline{\Omega}$ with $\gamma_2(0)=\gamma_2(T_2)$ such that
	\[
	\int_0^{T_2}L(\gamma_2,\dot{\gamma}_2)\ ds<0.
	\]
Thus, $\Phi^{\Omega}_L(x,y)=-\infty$ for all $x$, $y\in\overline{\Omega}$. This shows that, if (2) is satisfied, then (3) holds. On the other hand, if $\Phi^{\Omega}_L(x,y)=-\infty$ for all $x$, $y\in\overline{\Omega}$, then there is a closed curve $\gamma_3:[0,T_3]\to\overline{\Omega}$ with $\gamma_3(0)=\gamma_3(T_3)$ such that
	
	\[
	\int_0^{T_3}L(\gamma_3,\dot{\gamma}_3)\ ds<0,
	\]
which implies $c_{\Omega}(L)>0$. This completes the proof of the equivalence of (2) and (3).
\end{proof}

\section{Changing a Lagrangian by an exact 1-form}

\begin{Pro}\label{exact_form}
Let $\frac{\partial L}{\partial v}(x,0)$ be an exact 1-form, say $\frac{\partial L}{\partial v}(x,0)=DS(x)$ with some function $S$ of $C^3$ class on $\R^n$. We set
\begin{gather*}
		L_1(x,v)=L(x,v)-\langle DS(x),v\rangle,\quad (x,v)\in\R^n\times\R^n,\\
	 g_1(x)=g(x)-S(x),\quad x\in\partial\Omega,\\
H_1(x,p)=H(x,p+DS(x)),\quad (x,p)\in\R^n\times\R^n.
\end{gather*}
If $u$ is the value function of \eqref{eq:represent_formulae} with respect to $(L,g)$, then $u_1=u-S$ is the value function of \eqref{eq:represent_formulae} with respect to $(L_1,g_1)$ and $u_1$ is a solution of \eqref{eq:HJ_BP} with respect to $(H_1,g_1)$ in the viscosity sense. In particular, $c_{\Omega}(L_1)=c_{\Omega}(L)$ and the compatibility condition \eqref{eq:supcritical1} is also satisfied for $L_1$.
\end{Pro}

\begin{proof}
	We denote by $A^{\Omega,L_1}_t(x,y)$ the fundamental solutions with respect to $L_1$, and $\Phi^{\Omega}_{L_1}(x,y)$ the associated Ma\~n\'e's potential. For any $x,y\in\overline{\Omega}$ and $t>0$, there exists $\xi\in\Gamma^{0,t}_{x,y}(\overline{\Omega})$ such that
	\begin{align*}
		A^{\Omega,L_1}_t(x,y)=&\int^t_0L_1(\xi,\dot{\xi})\ ds=\int^t_0L(\xi,\dot{\xi})\ ds-\int^t_0\langle DS(\xi),\dot{\xi}\rangle\ ds\\
		=&\int^t_0L(\xi,\dot{\xi})\ ds-\int^t_0\left(\frac d{ds}S(\xi(s))\right) ds\\
		=&\int^t_0L(\xi,\dot{\xi})\ ds-(S(y)-S(x))\geqslant A^{\Omega}_t(x,y)-(S(y)-S(x)).
	\end{align*}
	The opposite direction of the inequality above can be obtained similarly. It follows
	\begin{equation}\label{eq:L_L1}
		\begin{split}
			A^{\Omega,L_1}_t(x,y)=&A^{\Omega}_t(x,y)-(S(y)-S(x)),\\
		\Phi^{\Omega}_{L_1}(x,y)=&\Phi^{\Omega}_L(x,y)-(S(y)-S(x)).
		\end{split}
	\end{equation}
	Therefore, for any $x\in\overline{\Omega}$,
	\begin{align*}
		u_1(x)=&\inf_{y\in\partial\Omega}\{g_1(y)+\Phi^{\Omega}_{L_1}(y,x)\}=\inf_{y\in\partial\Omega}\{g_1(y)+\Phi^{\Omega}_{L}(y,x)-S(x)+S(y)\}\\
		=&\inf_{y\in\partial\Omega}\{g(y)+\Phi^{\Omega}_{L}(y,x)\}-S(x)=u(x)-S(x).
	\end{align*}
	
	Finally, the relation $c_{\Omega}(L_1)=c_{\Omega}(L)$ is obvious since $A^{\Omega,L_1}_t(x,x)=A^{\Omega}_t(x,x)$ for any $x\in\overline{\Omega}$ and $t>0$ by \eqref{eq:L_L1} and the compatibility condition \eqref{eq:supcritical1} follows from \eqref{eq:L_L1} and the definition of $g_1$.
\end{proof}
}

\section*{Acknowledgements}
Piermarco Cannarsa is partly supported by the University of Rome Tor Vergata (Consolidate the Foundation 2015) and Istituto Nazionate di Alta Matematica (GNAMPA 2017 Research Projects). Wei Cheng is partly supported by Natural Scientific Foundation of China (Grant No. 11631006 and No.11790272). Kaizhi Wang is partly supported by National Natural Science Foundation of China (Grant No. 11771283).

%

\end{document}